\documentclass[11pt,a4paper]{amsart}

\usepackage{amsfonts,amsmath,amssymb,amsthm}

\usepackage[utf8]{inputenc}
\usepackage[T1]{fontenc}

\usepackage{tikz}
\usetikzlibrary{calc}

\usepackage{hyperref}

\usepackage{stmaryrd}

\newcommand{\Cc}{\mathbb{C}}

\newcommand{\Nn}{\mathbb{N}}
\newcommand{\Zz}{\mathbb{Z}}
\newcommand{\Qq}{\mathbb{Q}}
\newcommand{\Ff}{\mathbb{F}}
\newcommand{\Aa}{\mathbb{A}}

\renewcommand {\leq}{\leqslant}
\renewcommand {\geq}{\geqslant}
\renewcommand {\le}{\leqslant}
\renewcommand {\ge}{\geqslant}

\newcommand{\lcm}{\mathop{\mathrm{lcm}}\nolimits}
\newcommand{\Spec}{\mathop{\mathrm{Spec}}\nolimits}

\newcommand{\defi}[1]{\emph{#1}}

\newcommand{\ul}[1]{\underline{#1}}

{\theoremstyle{plain}
\newtheorem{theorem}{Theorem}[section]    
\newtheorem{lemma}[theorem]{Lemma}       
\newtheorem{proposition}[theorem]{Proposition}      
\newtheorem{corollary}[theorem]{Corollary}      
\newtheorem*{theorem*}{Theorem}
\newtheorem*{hypothesis*}{Schinzel Hypothesis}
}
{\theoremstyle{remark}
\newtheorem{remark}[theorem]{Remark}   
\newtheorem{example}[theorem]{Example}
}

\usepackage[a4paper]{geometry}
\makeatletter
\renewcommand{\pod}[1]{\allowbreak\mathchoice
  {\if@display \mkern 8mu\else \mkern 8mu\fi (#1)}
  {\if@display \mkern 8mu\else \mkern 8mu\fi (#1)}
  {\mkern4mu(#1)}
  {\mkern4mu(#1)}
}
\makeatother

\title{Coprime values of polynomials in several variables}

\author{Arnaud Bodin}
\author{Pierre D\`ebes}

\email{arnaud.bodin@univ-lille.fr}
\email{pierre.debes@univ-lille.fr}

\address{Universit\'e de Lille, CNRS, Laboratoire Paul Painlev\'e, 59000 Lille, France}



\thanks{\emph{Acknowledgment}. 
This work was supported in part by the Labex CEMPI  (ANR-11-LABX-0007-01) 
and by the ANR project ``LISA'' (ANR-17-CE40–0023-01).}

\date{\today}

\begin{document}

\begin{abstract}
Given two polynomials $P(\underline x)$, $Q(\underline x)$ in one or more variables and with integer coefficients, how does the property that they are coprime relate to their values $P(\underline n),  Q(\underline n)$ at integer points $\underline n$ being coprime?
We show that the set of all $\gcd (P(\underline n),  Q(\underline n))$ is stable under gcd and under lcm.
A notable consequence is a result of Schinzel: if in addition $P$ and $Q$ have no fixed prime divisor (i.e., no prime dividing all values $P(\underline n)$, $Q(\underline n)$), then $P$ and $Q$ assume coprime values at ``many'' integer points. 
Conversely we show that if ``sufficiently many'' integer points yield values that are coprime (or of small gcd) then the original polynomials must be coprime.
Another noteworthy consequence of this paper is a version ``over the ring''  of Hilbert's irreducibility theorem.
\end{abstract}

\dedicatory{Dedicated to Moshe Jarden on the occasion of his 80th birthday}

\maketitle



Let $P_1(\ul x),\ldots,P_s(\ul x) \in \Zz[\ul x]$ be $s \ge 2$ polynomials in $r\geq 1$ variables $\ul x= (x_1,\ldots,x_r)$ .
For $\ul n = (n_1,\ldots,n_r) \in \Zz^r$, we consider the corresponding values $P_i(\ul n)$.
Is there a connection between (a) $P_1(\ul x),\ldots,P_s(\ul x)$ being coprime as polynomials and (b) ``many'' of the values $P_1(\ul n),\ldots,P_s(\ul n)$ being coprime as integers?
Answers exist in both directions. 

\smallskip

Suppose that the polynomials $P_1(\ul x),\ldots,P_s(\ul x)$ are coprime and their values have no \emph{fixed divisors},
i.e., no prime number $p$ divides all $P_i(\ul n)$ (for all $i$, and all $\ul n$). Then it is true that for some $\ul n \in \Zz^r$, the integers $P_1(\ul n),\ldots,P_s(\ul n)$ are coprime: coprime polynomials assume coprime values. 
This is proved by Schinzel in  \cite{Sc02}; Ekedahl \cite{Ek} and Poonen \cite{Poo} even give, in the special case $s=2$, a formula for the density of the good $\ul n$; see Section \ref{ssec:poonen-intro} below, and also \cite{BDKN21} where Schinzel's result is extended to other rings than $\Zz$, including all UFDs and all Dedekind domains.

\smallskip

Here we put forward a more general property of polynomials that implies Schinzel's coprime conclusion. Set $d_{\ul{n}} = \gcd (P_1(\ul{n}),\ldots,P_s(\ul n))$, for $\ul n \in \Zz^r$, the gcd of the values. We  show, even without the fixed divisor assumption, that the set $\mathcal{D}$ of all these $d_{\ul n}$ is stable under gcd and lcm, i.e., is a lattice for the divisibility (Theorem \ref{th:main}); the quick proof that it yields Schinzel's theorem is in Section \ref{ssec:applications}.
This generalizes previous results  in one variable  \cite{BDN20}.

\smallskip

Regarding the Ekedahl--Poonen formula, we extend it to the case of $s\geq 2$ polynomials and to the situation that \emph{several} families of such polynomials are given (Section \ref{ssec:poonen-intro}).
We can then show  a version ``over the ring $\Zz$'' of Hilbert's Irreducibility Theorem (Theorem \ref{cor:HS}).

\smallskip

In the reverse direction, it is not true that if $P_1(\ul n),\ldots,P_s(\ul n)$ are coprime at one integer point $\ul n$ (or even at infinitely many) then the polynomials $P_1(\ul x),\ldots,P_s(\ul x)$ are coprime. 
However we show that the coprimality of $P_1(\ul x),\ldots,P_s(\ul x)$ does hold  if ``sufficiently many'' $\ul n$, in a density sense, can be found such that $P_1(\ul n),\ldots,P_s(\ul n)$ 
are coprime  (Theorem \ref {th:crit-coprime}).

\medskip
The Hilbertian specialization property has always been a central topic in Field Arithmetic. Through his work, 
Moshe Jarden has constantly promoted both the area and this subtopic. The celebrated ``Fried-Jarden book'', {\it the} Field 
Arithmetic reference, has been 
quite influential to both authors. With this paper, we are happy to contribute to the Israel Journal of Mathematics 
special volume dedicated to Moshe Jarden and to offer him as a final application a version ``over 
the ring'' of Hilbert's irreducibility theorem.

\section{Presentation}
\label{sec:presentation}

Throughout the paper, we adhere to the following notation. Given $s\ge2$ nonzero polynomials $P_1(\ul x),\ldots,P_s(\ul x)$ in $\Zz[\ul{x}]$ (where $\ul{x} = (x_1,\ldots,x_r)$ with $r\ge1$), we say that
they are \emph{coprime} (over the field $\Qq$) if no polynomial $D(\ul x) \in \Qq[\ul{x}]$ with $\deg D>0$ divides each of $P_1(\ul{x}),\ldots,P_s(\ul{x})$.
In the definition of $d_{\ul{n}} = \gcd (P_1(\ul{n}),\ldots,P_s(\ul n))$ ($\ul{n} \in \Zz^r$), we include the case where $P_1(\ul n)= \ldots =  P_s(\ul n)=0$ by defining $\gcd (0,\ldots,0) = 0$.
Finally we set $\mathcal{D} = \left\lbrace d_{\ul{n}} \mid \ul{n} \in \Zz^r \right\rbrace$.

\subsection{The stability result}

\begin{theorem}
\label{th:main}
If $P_1(\ul{x}),\ldots,P_s(\ul{x}) \in \Zz[\ul{x}]$ are $s\geq 2$ nonzero coprime polynomials, then the set $\mathcal{D} = \left\lbrace d_{\ul{n}} \mid \ul{n} \in \Zz^r \right\rbrace$ is stable under gcd and lcm.
\end{theorem}

That is: 
if $d,d' \in \mathcal{D}$ then $\gcd(d,d') \in \mathcal{D}$ and $\lcm(d,d') \in \mathcal{D}$.
This is a generalization of the one variable case ($r=1$) done with S.~Najib \cite{BDN20}.

\begin{example}
\label{ex:intro}
Let $P(x,y) = x^2 - y^3$, $Q(x,y) = x(y + 2)+1$. Let $d_{m,n} = \gcd( P(m,n), Q(m,n) )$ and $\mathcal{D} = \{ d_{m,n} \}_{m,n \in\Zz}$. 
For instance $P(5,1) = 24$, $Q(5,1)=16$, hence $d_{5,1} = \gcd(24,16) = 8$. For $(m,n)=(1,-3)$, $d_{m,n} = 28$.
The gcd of $8$ and $28$ is $4$, and $4$ is 
an element of $\mathcal{D}$: $d_{5,5} = 4$.
Experimentation yields an infinite set:
\begin{multline*}
\mathcal{D}
= \{ 1, 2, 4, 7, 8, 14, 16, 23, 28, 29, 32, 37, 41, 46, 47, 49, \\
53, 56, 58, 59, 61, 64, 67, 74, 79, 82, 83, 89, 92, 94, 97, 98,\ldots \}
\end{multline*}
\end{example}

\subsection{Consequences} 
\label{ssec:applications}

The following two corollaries are quick consequences of Theorem \ref{th:main}. The first one is what we refer to as Schinzel's result in our introduction.

\begin{corollary}
\label{cor:coprime}
Let $P_1(\ul{x}),\ldots,P_s(\ul{x}) \in \Zz[\ul x]$ be $s\geq 2$ nonzero coprime polynomials.
Suppose that there is no prime number $p$ that divides $P_i(\ul n)$ for each $i=1,\ldots,s$ and every $\ul n \in \Zz^r$.
Then there exists $\ul n_0 \in \Zz^r$ such that $P_1(\ul n_0),\ldots,P_s(\ul n_0)$ are coprime integers. 
\end{corollary}

Moreover, the set of such $\ul n_0$ will be shown to be Zariski-dense in $\Zz^r$ (Corollary \ref{cor:zariski}),
and even of positive density (as discussed in \S \ref{ssec:poonen-intro} below and shown in \S \ref{sec:proof-poonen}).  

\begin{proof}[Proof of Corollary \ref{cor:coprime} assuming Theorem \ref{th:main}]
The set $\mathcal{D}\subset \Nn$ is not necessarily finite (for $r\ge2$). Let $\{ d_{i_j} \}_{j\in\Nn}$ be an enumeration
of $\mathcal{D}^\star = \mathcal{D} \setminus \{0\}$
and set $\delta_j = \gcd(d_{i_0},\ldots,d_{i_j})$. The sequence $(\delta_j)_{j\in\Nn}$ is a decreasing sequence of positive integers, hence is ultimately constant equal to some value $d^\star \in \Nn$,
and $d^\star = \gcd(\mathcal{D}^\star) = \min(\mathcal{D}^\star)$.

By Theorem \ref{th:main}, $\mathcal{D}$ is stable by gcd; so is $\mathcal{D}^\star$. Using $\gcd(a,b,c) = \gcd(\gcd(a,b),c)$, we have $\delta_j \in \mathcal{D}^\star$, for every $j\in\Nn$.
It follows that $d^\star\in \mathcal{D}^\star$.
The no fixed divisor assumption yields $d^\star = 1$. Hence $1 \in \mathcal{D}^\star$, thus giving the conclusion.
\end{proof}

\begin{corollary} 
Let $P_1(\ul{x}),\ldots,P_s(\ul{x}) \in \Zz[\ul x]$ be $s\geq 2$ nonzero polynomials with no common zero in $\Cc^r$. 
Then $\mathcal{D}$ is a finite subset of $\Zz$ stable under gcd and lcm. 
In particular, the smallest positive element 
$d^\star$ of $\mathcal{D}$ is a common divisor of all elements of $\mathcal{D}$ and the largest positive element $\mu^\star$ of $\mathcal{D}$ is a common multiple of all elements of $\mathcal{D}$.
\end{corollary}

\begin{proof} 
Hilbert's Nullstellensatz
provides polynomials $A_1(\ul x),\ldots, A_s(\ul x) \in \Qq[x]$ such that $\sum_{i=1}^s A_i(\ul x) P_i(\ul x)=1$. Clearing the denominators yields polynomials $B_1(\ul x),\ldots, B_s(\ul x) \in \Zz[x]$ and $\Delta \in \Zz$, $\Delta\not= 0$, such that $\sum_{i=1}^s B_i(\ul x) P_i(\ul x)=\Delta$. It readily follows that every element $d_{\underline n} \in \mathcal{D}$ divides $\Delta$. Hence  
$\mathcal{D}$ is finite. The rest is given by Theorem \ref{th:main}.
\end{proof}

\subsection{Ekedahl--Poonen formula}
\label{ssec:poonen-intro}

Given $s\geq 2$ nonzero coprime polynomials $P_1(\ul{x}),\ldots,P_s(\ul{x}) \in \Zz[\ul x]$ as in Theorem \ref{th:main}, this formula
provides another refinement of Corollary \ref{cor:coprime}: it computes the density of integer points where the values are coprime.
Specifically let 
$$\textstyle \mathcal{R} = \big\{ \ul n \in \Zz^r \mid P_1(\ul n), \ldots,P_s(\ul n) \text{ are coprime} \big\}.$$

The \emph{density} $\mu(\mathcal{S})$ of a subset $\mathcal{S}$ of points with non-negative integer coordinates is defined as follows. 
For $B > 0$, set $\mathbb{B} = \llbracket 0,B-1\rrbracket^r$, where $\llbracket 0,B-1\rrbracket$ is the set of integers from $0$ to $B-1$. 
Then
$$\mu(\mathcal{S}) = \lim_{B \to +\infty} \frac{ \# (\mathcal{S}\cap \mathbb{B})}{\# \mathbb{B}}.$$

The sets we consider are subsets of $\Zz^r$ and our results 
are about their density within the $r$-dimensional quadrant 
$[0,+\infty[^r$. For simplicity of notation, we extend the 
definition of $\mu$ to subsets $S \subset \Zz^r$ 
by setting: $\mu (S) = \mu (S \cap ([0,+\infty[^r))$.
Remark \ref{rem:density} explains that, in addition to giving the 
density of $\mathcal{R}$, Theorem 1.5 shows that $\mathcal{R}$ is equidistributed 
among all $r$-dimensional quadrants.

Denote the set of prime numbers by $\mathcal{P}$.

\begin{theorem}[Ekedahl--Poonen density formula]
\label{th:poonen} 
Let $\ul x = (x_1,\ldots,x_r)$ ($r\ge1$).
Let $P_1(\ul x),\ldots,P_s(\ul x) \in \Zz[\ul x]$ ($s\ge2$) be nonzero coprime polynomials.
We have:
$$\mu(\mathcal{R}) = \prod_{p \in \mathcal{P}} \left( 1 - \frac{c_p}{p^r} \right)$$
where 
$c_p = \# \big\{ \ul n \in (\Zz/p\Zz)^r \mid P_1(\ul n) = 0 \pmod p, \ldots, P_s(\ul n) = 0 \pmod p \big\}$.
\end{theorem}

If we assume, as in Corollary \ref{cor:coprime}, that there is no prime $p$ dividing all values $P_1(\ul n),\ldots,P_s(\ul n)$ ($\ul n \in \Zz^r$), 
we obtain that {\it $\mathcal{R}$ is of positive density}: all terms in the product from Theorem \ref{th:poonen} are positive, and 
the product is convergent if $r\ge2$ and finite if $r=1$ (as shown in Section \ref{ssec:poonen-one}).

\begin{remark} \label{rem:density}
It follows from the formula for $\mu(\mathcal{R})$ that the density of $\mathcal{R}$ would be the same if computed w.r.t to any other $r$-dimensional quadrant, instead of $[0,+\infty[^r$: indeed the number $c_p$ of solutions of $P_1(\ul n) = \cdots = P_s(\ul n) = 0 \pmod p$ in a box of width $p$ is independent of the choice of the box. 
This also shows that for the density defined by 
$\tilde\mu(\mathcal{R}) = \lim_{B \to +\infty} \frac{ \# (\mathcal{R}\cap \mathbb{B})}{\# \mathbb{B}}$ with this time $\mathbb{B} = \llbracket -B,B \rrbracket^r$,  then 
$\tilde\mu(\mathcal{R}) = \mu(\mathcal{R})$.
\end{remark}

We provide a proof of the Ekedahl--Poonen formula in Section \ref{sec:proof-poonen}. It follows Poonen's proof with some adjustments. In particular we consider the general case $s \ge 2$ (and not just $s=2$).
We also consider in Section \ref{ssec:gen_sev_pol} the more general situation that \emph{several} families of coprime polynomials $\{P_{1i}(\ul x)\}_i$, $\{P_{2i}(\ul x)\}_i$,\ldots, $\{P_{\ell i}(\ul x)\}_i$ are given and one looks for the density of the set of points
$\ul n \in \Zz^r$ such that, for each $j=1,\ldots,\ell$, the integers $P_{j1}(\ul n),P_{j2}(\ul n),\ldots$ are coprime (Proposition \ref{prop:poonen-gen}). This generalization will be used to prove the case of \emph{several} polynomials in the following result.

\subsection{A version over the ring  of Hilbert's Irreducibility Theorem} \label{ssec:HS-intro}

\begin{theorem} \label{cor:HS} Let $\underline y= (y_1,\ldots,y_n)$ be $n\geq 1$ new variables.
Let $P_1(\underline x,\underline y), \ldots, P_\ell(\underline x,\underline y)$ be $\ell\geq 1$ polynomials, irreducible in $\Zz[\underline x,\underline y]$, of degree $\geq 1$ in $\underline y$. Assume there is no prime $p$ such that $\prod_{j=1}^\ell P_j(\underline n,\underline y) \equiv 0 \pmod{p}$ for every $\underline n \in \Zz^r$. Then the set  of all $\underline n\in \Zz^k$ such that $P_1(\underline n,\underline y),\ldots,P_\ell(\underline n,\underline y)$ are irreducible in $\Zz [\underline y]$ is Zariski-dense, \hbox{and even of positive $\tilde \mu$-density.} 
\end{theorem}

Here, for ``many'' $\underline n \in \Zz^r$, the specialized polynomials $P_1(\underline n,\underline y),\ldots,P_\ell(\underline n,\underline y)$ are irreducible in $\Zz [\underline y]$, and not only in $\Qq [\underline y]$ as Hilbert's Irreducibility Theorem would conclude: we have the additional conclusion that each polynomial $P_j(\underline n,\underline y)$ is primitive, i.e., its coefficients are coprime integers. The assumption on the product $\prod_{j=1}^\ell P_j$ is clearly necessary and non void: for $P(x,y)=(x^2-x)y+(x^2-x+2)$, we have $P(n,y) \equiv 0 \pmod{2}$ and so $P(n,y)$ is divisible by $2$ in $\Zz [y]$, for every $n \in \Zz$.

Theorem \ref{cor:HS} compares to Theorem 1.6 from \cite{BDKN21} (joint with Najib and König). The latter considers more general rings
(UFDs or Dedekind domains, with a product formula), but does not have the density conclusion provided here in the special
case of the ring of integers. The density approach also allows a quick proof of Theorem \ref{cor:HS} assuming Theorem  \ref{th:poonen}. 
The argument below is for $\ell = 1$; a reduction to this case is explained in Section \ref{ssec:gen_sev_pol}.

\begin{proof}
Set $P=P_1$ and
let ${\mathcal H}_P$ be the subset of $\Zz^r$ of all $\underline n$ such that $P(\underline n,\underline y)$ is irreducible in $\Qq [\underline y]$.
From Theorem 1 of \cite[\S 13]{serreMW} (a result of S.D. Cohen),  ${\mathcal H}_P$ is of density $\tilde\mu(\mathcal{{\mathcal H}_P})=1$ (with $\tilde\mu$ 
the density from Remark \ref{rem:density}). Denote the coefficients of $P$, viewed as a polynomial in $\underline y$,
by $P_1(\ul x),\ldots, P_s(\ul x)$ and consider the set $\textstyle \mathcal{R}$ from Section \ref{ssec:poonen-intro} of all $\underline n\in \Zz^r$ such that $P_1(\ul n),\ldots, P_s(\ul n)$ 
are coprime. The assumption of Theorem \ref{cor:HS} corresponds to $P_1(\ul x),\ldots, P_s(\ul x)$ having no fixed divisor. From Theorem \ref{th:poonen} and Remark \ref{rem:density}, 
we have $\tilde\mu(\textstyle \mathcal{R})>0$. It follows that  $H={\mathcal H}_P \cap \textstyle \mathcal{R}$ is of positive $\tilde \mu$-density, thus proving the result since for every $\underline n \in H$, the polynomial
$P(\underline n, \underline y)$ is irreducible in $\Zz[\underline y]$.
\end{proof}

\subsection{A criterion for coprimality}
\label{ssec:criterion}

In our introduction, we raised this reverse question:
to what extent existence of coprime values forces the coprimality of the polynomials?
For one variable polynomials we have this coprimality criterion involving the gcd in $\Zz$ of some values. 
Define the \defi{normalized height} of a degree $d$ polynomial 
$P(x)= a_dx^d+ \cdots + a_0$ by
$H(P) = \max_{i=0,\ldots,d-1} \left| \frac{a_i}{a_d} \right|.$

\begin{proposition}[\hbox{\cite[Proposition 5.1]{BDN20}}]
\label{prop:coprime}
Let $P_1,\ldots,P_s \in \Zz[x]$ be $s\geq 2$ nonzero polynomials and $H$ the minimum of the normalized heights $H(P_1),\ldots,H(P_s)$. Then $P_1,\ldots,P_s$ are coprime if and only if there exists $n\ge 2H+3$ such that $\gcd(P_1(n),\ldots,P_s(n)) \le \sqrt{n}$.
\end{proposition}

In particular if $P_1(n),\ldots,P_s(n)$ are coprime (as integers) for some sufficiently large $n$ then $P_1(x),\ldots,P_s(x)$ are coprime (as polynomials). 
We wish to generalize this result to polynomials in several variables. But the following example proves that evaluation at one point, however big it is, may not give information on the coprimality of the polynomials: 
with $P(x,y) = (x-y)x$ and $Q(x,y) = (x-y)y$, we have $\gcd( P(n+1,n), Q(n+1,n) ) = 1$,
and so infinitely many points $(n+1,n)$ where the gcd is small, despite the polynomials not being coprime.

The following result however ensures that if the gcd $d_{\ul n}$ is small for ``sufficiently many'' $\ul n$, in a stronger density sense, then the polynomials are coprime.

\begin{theorem}
\label{th:crit-coprime}
Let $P_1(\ul{x}),\ldots,P_s(\ul{x}) \in\Zz[\ul x]$ be $s\geq 2$ nonzero polynomials in $r$ variables. Let $\ell = \max (\deg P_1,\ldots,\deg P_s)$ and $S$ be a nonempty finite set of $\Zz$. Let $k>0$.
If
$$\pi_k := \frac{\# \left\{\ul{n} \in S^r  \mid d_{\ul n} \le k \right\}}{\#S^r}> \frac{(2k+1)\ell}{\# S}$$
then $P_1(\ul {x}),\ldots,P_s(\ul {x})$ are coprime polynomials.
\end{theorem}
For $k=1$, we have 
$\pi_1 := \frac{\# \left\{  \ul{n} \in S^r \mid d_{\ul n} \le 1\right\}}{\#S^r}$.
Theorem \ref{th:crit-coprime} states that if $\pi_1 > \frac{3\ell}{\# S}$
then $P_1(\ul {x}),\ldots,P_s(\ul {x})$ are coprime polynomials; and clearly this also implies that $P_1(\ul{x}),\ldots,P_s(\ul{x})$ have no fixed prime divisor.
This criterion is of interest because of the Ekedahl--Poonen density formula.
If polynomials $P_1(\ul{x}),\ldots,P_s(\ul{x})$ are coprime  and have no fixed  prime divisor, then $\pi_1$ must be positive for sufficiently large $S$, and so up to taking $S$ large enough, the criterion will indeed reach the coprimality conclusion.

\begin{example}
Let $P(x,y), Q(x,y)\in \Zz[x,y]$ be two nonzero polynomials of degree $\leq \ell:=10$.
Let $S=\{1,2,\ldots,100\}$ with $\#S=100$.
If for more than $30\%$ of $(m,n) \in S^2$, we have 
$d_{m,n} = 1$ (i.e., $P(m,n)$ and $Q(m,n)$ coprime)
or $d_{m,n} = 0$ (i.e., $P(m,n)=Q(m,n)=0$), 
then we have $\pi_1 > \frac{30}{100}$, and so, from Theorem \ref{th:crit-coprime}, $P(x,y)$ and $Q(x,y)$ are coprime polynomials.
\end{example}

\begin{proof}[Proof of Theorem \ref{th:crit-coprime}]
It relies on the Zippel--Schwartz lemma which is usually stated as a probability result, but in fact is an enumerative result.

\smallskip

\noindent
\textbf{Zippel--Schwartz lemma.} \emph{Let $P(x_1,\ldots,x_r)$ be a nonzero polynomial of degree $\ell$ over a field $K$.
Let $S$ be a nonempty finite set of $K$. 
Then 
$$\frac{\# \left\{ (x_1,\ldots,x_r) \in S^r \mid P(x_1,\ldots,x_r)=0 \right\}}{\#S^r} \le \frac{\ell}{\# S}.$$}

Let $D(\ul{x}) = \gcd(P_1(\ul{x}), \ldots,P_s(\ul{x}))$.
Then $\deg D \le \ell$.
Note further that $D(\ul{n})$ divides $d_{\ul{n}} = \gcd(P_1(\ul{n}), \ldots,P_s(\ul{n}))$, so that
$|D(\ul{n})| \le d_{\ul{n}}$.
Now assume, by contradiction, that $D$ is a non constant polynomial.
We use the Zippel--Schwartz lemma to bound the number of solutions to the equations $D(\ul{n}) = j$. 
Specifically we have:
\begin{align*}
\pi_k 
&= \frac{\# \left\{ \ul{n} \in S^r \mid d_{\ul n} \le k \right\}}{\#S^r} 
\le \frac{\# \left\{  \ul{n} \in S^r  \mid |D(\ul{n})| \le k \right\}}{\#S^r} \\
&\le \sum_{j=-k}^{k} \frac{\# \left\{  \ul{n} \in S^r  \mid D(\ul{n}) = j \right\}}{\#S^r} 
\le (2k+1) \frac{\ell}{\# S} \qedhere
\end{align*}

\end{proof}

\smallskip

The paper is organized as follows. In Section \ref{sec:one}, 
we focus on the case of polynomials in one variable. 
In Section \ref{sec:specialization}, we present a tool of frequent use in the paper about how coprimality is preserved by specialization, 
in the vein of the Bertini--Noether and Ostrowski theorems for irreducibility (Proposition \ref{prop:bertini-intro}).
Section \ref{sec:lemma} is devoted to a technical lemma, used in Section \ref{sec:proof} for the proof of Theorem \ref{th:main}.
We end in Section \ref{sec:proof-poonen} with a proof of the Ekedahl--Poonen formula in the case of several polynomials.

\section{The one variable case}
\label{sec:one}

The case of one variable polynomials plays a central role: first, some of the general results can be interestingly improved; 
secondly, most results in several variables will follow by reduction from the one variable case.

\subsection{Stability by gcd and lcm}

\begin{theorem}[\hbox{\cite[Prop. 3.2 and 3.3]{BDN20}}]
\label{th:stableone}
Let $P_1(x),\ldots,P_s(x) \in\Zz[x]$ be nonzero coprime polynomials.
Set $d_n = \gcd (P_1(n),\ldots,P_s(n))$ ($n\in \Zz$). 
Then the set
$\mathcal{D} = \big\lbrace d_n \mid n \in \Zz \big\rbrace$
is stable under gcd and lcm.
Moreover there is a nonzero $\delta\in \Zz$ that is a common multiple  to all $d_n$  and such that the sequence $(d_n)_ {n\in\Zz}$ 
is periodic of period $\delta$.
Hence $\mathcal{D}$ is a finite set.
\end{theorem}

As $P_1(x),\ldots,P_s(x)$ are coprime, note that it cannot happen that 
$P_1(n) = \ldots = P_s(n) = 0$.
The periodicity result is specific to the one variable case (see \cite[\S 2.5]{BDN20}); $\delta$ can be taken to be any nonzero element of the ideal $\langle P_1,\ldots,P_s\rangle \cap \Zz\subset \Zz[\ul x]$.
For two polynomials $P(x)$ and $Q(x)$, $\delta$ can be chosen as the resultant of $P$ and $Q$. 
More generally, 
as the polynomials $P_1(x),\ldots,P_s(x)$ are coprime in $\Qq[x]$, one can write a Bézout identity:
$$A_1(x)P_1(x) + \cdots + A_s(x)P_s(x) = 1$$
for some  $A_1(x),\ldots,A_s(x)\in\Qq[x]$. Then $\delta$ can be taken to be the right-hand side of the identity obtained by clearing the denominators of the coefficients 
of the $A_i(x)$:  for some $B_1(x),\ldots, B_s(x)\in\Zz[x]$, we have $B_1(x)P_1(x) + \cdots + B_s(x)P_s(x) = \delta \in \Zz$.

\begin{example}
\label{ex:cex}
Theorem \ref{th:stableone} is false for non coprime polynomials. Let $P(x) = 5(x^2-1)(x-1)$ and $Q(x) = (x^2-1)x^2$.
Then $\mathcal{D}$ is an infinite set (because $d_n = \gcd(P(n),Q(n)) \ge |n^2-1|$ tends to infinity as $n\to+\infty$).
The set $\mathcal{D}$ is not stable by gcd:
for instance $d_2=3 \in \mathcal{D}$ and $d_6=8\in\mathcal{D}$, but $1 \notin \mathcal{D}$ (by contradiction, suppose that for some $n\in\Zz$ we have $d_n=1$,
then $|n^2-1|=1$, so $n=0$, but for $n=0$, $P(n)=5$, $Q(n) = 0$ and $d_n=5$).
Neither $\mathcal{D}$ is stable by lcm: $5\in\mathcal{D}$, $8\in\mathcal{D}$ but $40\notin\mathcal{D}$ (for $|n|<7$ we have $d_n\neq 40$ and for $|n|\ge7$, $d_n\ge|n^2-1|>40$).
\end{example}

\subsection{Proof of Theorem \ref{th:stableone}}

Everything in Theorem \ref{th:stableone} is proved in \cite{BDN20}, except the stability under lcm that was left to the
reader (after the proof for the gcd was given). For completeness we detail it here.

Let $d_{n_1}$ and $d_{n_2}$ be two elements of $\mathcal{D}$ and let $m(n_1,n_2)$ be their lcm. The goal is to prove that $m(n_1,n_2)$ is an element of $\mathcal{D}$. The integer $m(n_1,n_2)$ can be factorized:
$$m(n_1,n_2) = \prod_{i\in I} p_i^{\alpha_i}$$
where, for each $i\in I$,  $p_i$ is a prime divisor of $\delta$ (see Theorem \ref{th:stableone}) and $\alpha_i \in \Nn$ (maybe $\alpha_i = 0$ for some $i\in I$).

Fix $i\in I$. As $p_i^{\alpha_i}$ divides $m(n_1,n_2)$, then $p_i^{\alpha_i}$ 
divides $d_{n_1}$ or divides $d_{n_2}$;
say that $p_i^{\alpha_i}$ divides $d_{m_i}$ with $m_i$ equals $n_1$ or $n_2$. 

The Chinese Remainder Theorem provides an integer $n$, such that
$$n = m_i \pmod{p_i^{\alpha_i+1}} \quad \text{ for each } i\in I.$$

By definition, $p_i^{\alpha_i}$ divides $d_{n_1}$ or $d_{n_2}$, so $p_i^{\alpha_i}$ divides all integers $P_1(n_1),\ldots, P_s(n_1)$, or divides all integers $P_1(n_2),\ldots, P_s(n_2)$, so that $p_i^{\alpha_i}$ divides all $P_1(m_i),\ldots, P_s(m_i)$. As for each $j=1,\ldots, s$, $P_{j}(n) = P_{j}(m_i) \pmod{p_i^{\alpha_i}}$, we obtain that $p_i^{\alpha_i}$ also divides $P_1(n),\ldots, P_s(n)$. Whence $p_i^{\alpha_i}$ divides $d_n$ for each $i\in I$.

On the other hand $p_i^{\alpha_i+1}$ does not divide $d_{n_1}$ nor $d_{n_2}$. In particular
$p_i^{\alpha_i+1}$ does not divide $d_{m_i}$. Hence there exists $j_0 \in \{1,\ldots,s\}$ such that $p_i^{\alpha_i+1}$ does not divide $P_{j_0}(m_i)$.
As $P_{j_0}(n) = P_{j_0}(m_i) \pmod{p_i^{\alpha_i+1}}$, then $p_i^{\alpha_i+1}$ does not divide $P_{j_0}(n)$. Hence $p_i^{\alpha_i+1}$ does not divide $d_n$.

We have proved that $p_i^{\alpha_i}$ is the greatest power of $p_i$ dividing $d_n$, for every $i\in I$. As $d_n$ divides $\delta$, each prime factor of $d_n$ is one of the $p_i$ with $i\in I$. Conclude that $m(n_1,n_2) = d_n$.

\subsection{Ekedahl--Poonen density formula in one variable}
\label{ssec:poonen-one}

One main question is to decide if $d_n=1$ for some value $n\in\Zz$.
In Section \ref{ssec:poonen-intro}, we discussed the Ekedahl--Poonen density formula for any number $r$ of variables. For $r=1$, it is an exact formula.

\begin{proposition}
\label{th:poonen-one}
Let $P_1(x),\ldots,P_s(x) \in \Zz[x]$ be nonzero coprime polynomials. Let $\delta\in \Zz$ be a positive period of $(d_n)_{n\in\Zz}$. The number of $n\in \Zz$ with $0 \le n < \delta$ such that $d_n=1$ is
$$\delta \prod_{p | \delta} \left( 1 - \frac{c_p}{p} \right)$$
where $c_p$ is the number of $n \in \Zz/p\Zz$ such that $P_i(n) = 0 \pmod p$ for each $i=1,\ldots,s$. 
\end{proposition}

Note that, \emph{in the one variable case, $c_p=0$ for all sufficiently large primes $p$}. Namely let $\delta$ be a nonzero element of the ideal $\langle P_1,\ldots,P_s\rangle \cap \Zz\subset \Zz[\ul x]$. Thus $\delta$ is of the form
$\delta = B_1(x)P_1(x) + \cdots + B_s(x)P_s(x)$ for some $B_1,\ldots,B_s \in \Zz[\ul x]$. Clearly, if $p$ does not divide $\delta$, then $p$ does not divide $\gcd(P_1(n),\ldots,P_s(n))$ for any $n\in\Zz$, hence $c_p=0$. 
\smallskip

The \emph{proof of Proposition \ref{th:poonen-one} assuming Theorem \ref{th:poonen}} easily follows. For $r=1$, the density formula from Theorem \ref{th:poonen} is a finite product: $\mu(\mathcal{R}) = \prod_{p | \delta} \left( 1 - \frac{c_p}{p} \right)$. As the sequence $(d_n)_{n\in\Zz}$ is periodic of period $\delta$ (Theorem \ref{th:stableone}), the claimed exact formula follows, for $\delta$ equal to the specific element of $\Zz$ introduced above, or equal
to any positive period.

\begin{example}
For two polynomials we recover a formula of \cite{FP}: \emph{If $P(x), Q(x) \in \Zz[x]$ are two monic coprime polynomials with a square-free resultant $R$, then}
$$\# \left\{  n \in \llbracket 0, R-1 \rrbracket \mid d_n = 1\right\} = \prod_{p | R} (p-1) .$$ 
In fact, for two polynomials, the integer $\delta$ can be chosen to be $R$. And if $R$ is square-free, then $c_p = 1$ for all $p|R$ (see \cite[proof of Theorem 6]{FP}).
\end{example}

\section{A Bertini--Noether--Ostrowski property for coprimality }
\label{sec:specialization}

Proposition \ref{prop:bertini-intro} below is of frequent use in this paper. It explains how coprimality of polynomials is preserved 
by specialization. It is obtained in Section \ref{ssec:pf-bertini-intro} as a special case of Proposition \ref{prop:bertini}, which is 
an analog for coprimality of the Bertini--Noether theorem for irreducibility of polynomials (\hbox{e.g.} \cite[Prop.9.4.3]{FrJa}). This more general
result is stated and proved in Section  \ref{ssec:cop-red}. Section  \ref{ssec:O-B-N}  shows another standard special case concerned 
with reduction modulo $p$ (Corollary \ref{cor:ostrowski}), which will be used later in the proof of Corollary \ref{cor:bezout}.

\subsection{Specialization and coprimality}

\begin{proposition}
\label{prop:bertini-intro}
Let $k$ be an infinite field and $P_1(\ul a, \ul x),\ldots, P_s(\ul a, \ul x) \in k[\ul a, \ul x]$ be polynomials in the variables $\ul a = (a_1,\ldots,a_m)$ and $\ul x = (x_1,\ldots,x_r)$ (with  $s\ge2$, $m\ge1$, $r\ge1$).
The following conditions are equivalent:
	\begin{enumerate}
		\item[(i)] The gcd of $P_1(\ul a, \ul x),\ldots, P_s(\ul a, \ul x) \in k[\ul a, \ul x]$ is in $k[\ul a]$.
		\item[(ii)] The polynomials $P_1(\ul a, \ul x),\ldots, P_s(\ul a, \ul x) \in k[\ul a, \ul x]$ are coprime in $k(\ul a)[\ul x]$.
		\item[(iii)] There exists a proper Zariski-closed subset $Z$ of $k^m$ such that for all 
		$\ul a^\star \in k^m\setminus Z$, the polynomials $P_1(\ul a^\star, \ul x),\ldots, P_s(\ul a^\star, \ul x)$  are coprime in $k[\ul x]$.
		\item[(iv)] There exists a Zariski-dense subset $Y$ of $k^m$ such that for all 
		$\ul a^\star \in Y$, the polynomials $P_1(\ul a^\star, \ul x),\ldots, P_s(\ul a^\star, \ul x)$  are coprime in $k[\ul x]$.
	\end{enumerate}
\end{proposition}

\subsection{Coprimality and reduction} \label{ssec:cop-red}
Given an integral domain $Z$ and an ideal $\mathfrak{p}\subset Z$,
we denote by $\overline z^\mathfrak{p}$ the coset of an element $z\in Z$ modulo $\mathfrak{p}$; we use the
same notation for the induced reduction morphisms, e.g. on polynomial rings over $Z$. If $\mathfrak{p}\subset Z$
 is a prime ideal, we write $k^\mathfrak{p}$ for the fraction field of the integral domain $Z/\mathfrak{p}$. 
 
\begin{proposition} 
	\label{prop:bertini} 
	Let $Z$ be a Unique Factorization Domain (UFD) with fraction field $Q$, let $\ul x = (x_1,\ldots,x_r)$ be $r\ge1$ variables
and let $P_1(\ul x),\ldots, P_s(\ul x) \in Z[\ul x]$ be $s\ge2$ nonzero polynomials. Suppose also given a Zariski-dense subset ${\mathcal P} \subset \Spec  Z$
\footnote{
The subset ${\mathcal P} \subset \Spec Z$ only appears in condition (iv) below.  The assumption that ${\mathcal P}$ is Zariski-dense means that for every nonzero element $R\in Z$, 
there is a prime ideal ${\mathfrak{p}}\in {\mathcal P}$ such that $\overline R^{\mathfrak{p}} \not=0$. This is clearly necessary for ${\rm(iv)}$ to hold. In fact ${\rm(iv)}$ reformulates 
as saying that, with ${\mathcal C} \subset \Spec  Z$ the set of primes $\mathfrak{p}$ such that $\overline{P_1}^{\mathfrak{p}}(\ul x),\ldots,\overline{P_s}^{\mathfrak{p}}(\ul x)$ are coprime in 
$k^{\mathfrak{p}}[\ul x]$, the set ${\mathcal C} \cap {\mathcal P}$ is Zariski-dense in $\Spec  Z$. 
In the same vein, condition (iii) means that ${\mathcal C}$ contains a nonempty Zariski-open subset of $\Spec Z$.
}. Then the following five conditions are equivalent:
	\begin{enumerate}
		\item[(i)] The gcd in $Z[\ul x]$ of $P_1(\ul x),\ldots, P_s(\ul x)$  is in $Z$.
		\item[(ii)] $P_1(\ul x),\ldots, P_s(\ul x)$  are coprime in $Q[\ul x]$.
		\item[(iii)] There is a nonzero element $R_0\in Z$ with this property: for every prime ideal $\mathfrak{p} \subset Z$ such that $\overline R_0^{\mathfrak{p}} \not=0$, the polynomials $\overline{P_1}^{\mathfrak{p}}(\ul x),\ldots,\overline{P_s}^{\mathfrak{p}}(\ul x)$  are coprime in $k^{\mathfrak{p}}[\ul x]$.
		\item[${\rm(iv)}$] For every nonzero element $R\in Z$, there exists a prime ideal ${\mathfrak{p}}\in {\mathcal P}$ such that $\overline R^{\mathfrak{p}} \not=0$ 
		and  the polynomials $\overline{P_1}^{\mathfrak{p}}(\ul x),\ldots,\overline{P_s}^{\mathfrak{p}}(\ul x)$ are coprime in $k^{\mathfrak{p}}[\ul x]$.
		\item[(v)] For every nonzero element $R\in Z$, there exists a maximal ideal ${\mathfrak{p}} \subset Z$ such that $\overline R^{\mathfrak{p}} \not=0$ and  the polynomials $\overline{P_1}^{\mathfrak{p}}(\ul x),\ldots,\overline{P_s}^{\mathfrak{p}}(\ul x)$ are coprime in $k^{\mathfrak{p}}[\ul x]$.
	\end{enumerate}
	\end{proposition}

\begin{proof}[Proof of Proposition \ref{prop:bertini}]
(ii) $\implies$ (i). Assume on the contrary that the gcd, say $D(\ul x) \in Z[\ul x]$, of $P_1(\ul x),\ldots, P_s(\ul x)$ is not in $Z$. Then $D(\ul x)$ is of degree $\geq 1$ (so not a unit
of  $Q[\ul x]$) and is a common divisor of $P_1(\ul x),\ldots, P_s(\ul x)$ in $Q[\ul x]$. This contradicts (ii).

\smallskip

(i) $\implies$ (ii). Assume on the contrary that $P_1(\ul x),\ldots, P_s(\ul x)$ are not coprime in $Q[\ul x]$, i.e., a nonconstant polynomial $D(\ul x)\in  Q[\ul x]$ divides all $P_i(\ul x)$
in $Q[\ul x]$. We may assume that $D$ is in $Z[\ul x]$, and even, using that $Z[\ul x]$ is a UFD, that $D$ is irreducible in $Z[\ul x]$.
Write $P_i(\ul x) = D(\ul x) P_i'(\ul x)$
	with $P_i' \in Q[\ul x]$, $i=1,\ldots,s$. Clearing the denominators, one obtains 
polynomial equalities in $Z[\ul x]$:
	$q_i  P_i(\ul x) = D(\ul x) \tilde P_i'(\ul x)$,
	with $\tilde P_i' \in Z[\ul x]$ and $q_i \in Z$, $q_i\not=0$, $i=1,\ldots,s$.	
	 It follows that $q_i$ divides $\tilde P_i'$ in $Z[\ul x]$, $i=1,\ldots,s$, and so that $D$ is a common divisor in $Z[\ul x]$ of all the $P_i(\ul x)$.
	This contradicts (i).
	

\begin{remark} \label{rk:bertini} (a) The equivalence ${\rm(i)} \Leftrightarrow{\rm(ii)}$ has this close variant: 
\vskip 0,7mm

\noindent
\emph{$P_1(\ul x),\ldots, P_s(\ul x)$ are coprime polynomials in $Z[\ul x]$ if and only if the equivalent conditions {\rm (i)}, {\rm (ii)} hold \ul{and} the coefficients of $P_1(\ul x),\ldots, P_s(\ul x)$ are coprime in $Z$}. 
\vskip 0,8mm

\noindent
Indeed, if $P_1(\ul x),\ldots, P_s(\ul x)$ are coprime in $Z[\ul x]$, they are coprime in $Q[\ul x]$ (by ${\rm(i)} \Rightarrow {\rm(ii)}$), and obviously, the coefficients of 
$P_1(\ul x),\ldots, P_s(\ul x)$ must be coprime in $Z$. Conversely, if $P_1(\ul x),\ldots, P_s(\ul x)$ are coprime in $Q[\ul x]$ and their coefficients are coprime in $Z$, then their gcd in $Z[\ul x]$ is in $Z$
(by ${\rm(ii)}\Rightarrow {\rm(i)}$), so must necessarily be $1$.
\end{remark}

(iii) $\implies$ ${\rm(iv)}$. For a given nonzero element $R\in Z$, let ${\mathfrak{p}} \in {\mathcal P}$ be a prime ideal such that $\overline{RR_0}^{\mathfrak{p}} \not=0$, where $R_0\in Z$ is the nonzero element
given by (iii); such a ${\mathfrak{p}}$ exists as ${\mathcal P}$ is assumed to be Zariski-dense.
Then $\overline R^{\mathfrak{p}} \not=0$ and $\overline R_0^{\mathfrak{p}} \not=0$, and by (iii), the latter gives that 
$\overline{P_1}^{\mathfrak{p}}(\ul x),\ldots,\overline{P_s}^{\mathfrak{p}}(\ul x)$ are coprime in $k^{\mathfrak{p}}[\ul x]$.

\smallskip

${\rm(iv)}$ $\implies$ (i). Assume that the gcd, say $D(\ul x) \in Z[\ul x]$, of $P_1(\ul x),\ldots, P_s(\ul x)$ is a polynomial
	of degree $\geq 1$. Let $R\in Z$ be a nonzero coefficient of a monomial of degree $\geq 1$ of $D(\ul x)$. Then for every prime ideal 
	${\mathfrak{p} \in \mathcal P}$ such that $\overline R^{\mathfrak{p}} \not=0$, the reduced polynomial $\overline D^{\mathfrak{p}}(\ul x)$ is of degree $\geq 1$
	and is a common divisor of $\overline{P_1}^{\mathfrak{p}}(\ul x),\ldots, \overline{P_s}^{\mathfrak{p}}(\ul x)$ in $k^{\mathfrak{p}}[\ul x]$. This contradicts ${\rm(iv)}$.
			
	\smallskip
	
(ii) $\implies$ (iii). We proceed by induction on the number of variables $r\ge1$. 
	
	\emph{1st case:} $r=1$, i.e. $\ul x$ is a single variable $x$. The assumption (ii) that the polynomials $P_1(x),\ldots,P_s(x)$ are coprime in the
	Principal Ideal Domain (PID) $Q[x]$ provides a Bézout identity, which after clearing the denominators, is of this form:
	
		$$\sum_{i=1}^s A_i(x)P_i(x) = R_0$$
   with $A_1,\ldots,A_s \in Z[x]$ and $R_0\in Z, R_0\not=0$.
		
	Clearly then, for every prime ideal ${\mathfrak{p}} \subset Z$ such that $\overline R_0^{\mathfrak{p}} \not=0$, the reduced polynomials 
	$\overline{P_1}^{\mathfrak{p}}(\ul x),\ldots,\overline{P_s}^{\mathfrak{p}}(x)$  satisfy a Bézout identity in the PID $k^{\mathfrak{p}}[x]$, hence are 
	coprime in $k^{\mathfrak{p}}[x]$.
		
\smallskip

	\emph{2nd case:} $r\geq 2$. Let $\ul x = (x_1,\ldots,x_{r-1},x_r)$ and assume that  (ii) $\Rightarrow$ (iii) is true for polynomials in the $r-1$ variables $(x_1,\ldots,x_{r-1})$.
	We will apply the induction hypothesis to the set of all coefficients $P_{i,j}(x_1,\ldots,x_{r-1})$ of the polynomials  $P_i(x_1,\ldots,x_r)$ viewed as polynomials 
	in $x_r$.

	The polynomials  $P_1(\ul x),\ldots, P_s(\ul x)$ are supposed to be coprime in $Q[x_1,\ldots,x_r]$. Thus, 
	by the already proven implication ${\rm(i)} \Rightarrow {\rm(ii)}$ (applied with $Z$ being $Q[x_1,\ldots,x_{r-1}]$), they are coprime in $Q(x_1,\ldots,x_{r-1})[x_r]$,
	and their coefficients $P_{i,j}(x_1,\ldots,x_{r-1})$ are coprime in $Q[x_1,\ldots,x_{r-1}]$. The former condition 
	provides a Bézout identity, which after clearing the denominators, is of this form:
	
		$$\sum_{i=1}^s A_i(\ul x)P_i(\ul x) = \Delta(x_1,\ldots,x_{r-1})$$
		
	with $A_1,\ldots,A_s \in Z[\ul x]$ and $\Delta \in Z[x_1,\ldots,x_{r-1}]$, $\Delta\not=0$.	
	Let $R_1\in Z$ be a nonzero coefficient of a monomial of $\Delta$. For every prime ideal 
	${\mathfrak{p}} \subset I$ such that $\overline R_1^{\mathfrak{p}} \not=0$, the polynomial $\overline \Delta^{\mathfrak{p}}(\ul x)$ is nonzero 
	in $k^{\mathfrak{p}}[x_1,\ldots,x_{r-1}]$, and so, the polynomials $\overline{P_1}^{\mathfrak{p}}(\ul x), \ldots, \overline{P_s}^{\mathfrak{p}}(\ul x)$ are coprime in $k^{\mathfrak{p}}(x_1,\ldots,x_{r-1})[x_r]$.
	
	Furthermore, as the coefficients $P_{i,j}(x_1,\ldots,x_{r-1})$
	are coprime in $Q[x_1,\ldots,x_{r-1}]$, the induction hypothesis provides a nonzero element $R_2\in Z$ such that
	for every prime ideal ${\mathfrak{p}} \subset I$ such that $\overline R_2^{\mathfrak{p}} \not=0$, the polynomials
	$\overline{P_{i,j}}^{\mathfrak{p}}(x_1,\ldots,x_{r-1})$ are coprime in $k^{\mathfrak{p}}[x_1,\ldots,x_{r-1}]$.
	
	Using the already proven implication ${\rm(ii)} \Rightarrow {\rm(i)}$ (more exactly its variant from Remark \ref{rk:bertini}), it  follows that the element $R_0=R_1R_2$ 
	satisfies the requested conclusion (iii).
\medskip	

{\it Equivalence of {\rm (v)} with all other conditions.} This follows from the fact that (v) is the special case of (iv) for which ${\mathcal P}$ is the set of all maximal ideals of $Z$. 
This subset ${\mathcal P} \subset \Spec Z$ is indeed Zariski-dense: as $Z$ is an integral domain, the nilradical ${\rm nil}(Z)$ (consisting of all nilpotent elements of $Z$) is $\{0\}$. But ${\rm nil}(Z)$ is classically the intersection of all maximal ideals of $Z$. Thus if $R\in Z$, $R\not=0$, there is a prime ideal ${\mathfrak{p}} \in {\mathcal P}$ 
such that $\overline R^{\mathfrak{p}} \not=0$ (which is the definition of ${\mathcal P}$ being Zariski-dense in $\Spec  Z$).
\end{proof}


\subsection{Proof of Proposition  \ref{prop:bertini-intro}} \label{ssec:pf-bertini-intro}
Proposition  \ref{prop:bertini-intro} corresponds to the special case of Proposition \ref{prop:bertini} for which $Z = k[\ul a]$ is a polynomial ring in $m\geq 1$ variables $\ul a = (a_1,\ldots,a_m)$
over a field $k$. Equivalence ${\rm (i)} \Leftrightarrow {\rm (ii)}$ from Proposition \ref{prop:bertini} exactly yields equivalence ${\rm (i)} \Leftrightarrow {\rm (ii)}$  from Proposition 
\ref{prop:bertini-intro} in this special case; the field $k$ need not be infinite here.

Assume now that $k$ is infinite and take for ${\mathcal P}$ the set of maximal ideals of the form $\langle \ul a - \ul a^\star\rangle= \langle a_1-a_1^\star,\ldots,a_r-a_r^\star\rangle$ with $\ul a^\star \in k^m$. With $k$ infinite, the subset ${\mathcal P} = \Aa^n(k)$ is indeed Zariski-dense. Condition (iv) from Proposition \ref{prop:bertini} then yields condition (iv) from Proposition \ref{prop:bertini-intro}.

Finally note that condition (iii) from Proposition \ref{prop:bertini} implies condition (iii) from Proposition \ref{prop:bertini-intro}, which itself implies condition (iv) from Proposition \ref{prop:bertini}, and so all three conditions are equivalent, thus ending the proof of Proposition \ref{prop:bertini-intro}. 

\subsection{The Ostrowski corollary} \label{ssec:O-B-N} For $Z=\Zz$, Proposition \ref{prop:bertini} yields the following result,
which is the coprimality analog of the Ostrowski theorem for irreducibility.

\begin{corollary} \label{cor:ostrowski}
Let  $P_1(\ul x),\ldots, P_s(\ul x) \in \Zz[\ul x]$ be $s\geq 2$ nonzero polynomials. The following four conditions
are equivalent.

\begin{enumerate}
	\item[(i)] The gcd of the polynomials $P_1(\ul x),\ldots, P_s(\ul x)$ is in $\Zz$.

	\item[(ii)] The polynomials $P_1(\ul x),\ldots, P_s(\ul x) \in \Zz[\ul x]$ are coprime in $\Qq[\ul x]$.

     \item[(iii)] For all but finitely many primes $p\in \Zz$, $\overline{P_1}^p(\ul x),\ldots, \overline{P_s}^p(\ul x)$ are coprime in $\Zz/p\Zz[\ul x]$.

     \item[(iv)] For infinitely many primes $p\in \Zz$, $\overline{P_1}^p(\ul x),\ldots, \overline{P_s}^p(\ul x)$ are coprime in $\Zz/p\Zz[\ul x]$.
\end{enumerate}

\end{corollary}

\begin{example}
How big should a prime number $p$ be to guarantee that two polynomials in $\Zz[\ul x]$ that are coprime in $\Qq[\ul x]$ remain coprime modulo $p$?
In the one variable case, it suffices that the prime $p$ does not divide the resultant of the two polynomials (which can be quite large). Here is an example in two variables.
Let $P(x,y) = x^3y - 3x^3 - 2x + 3y + 2$
and $Q(x,y) = y(2x-11)$. These polynomials are coprime in $\Zz[x,y]$.
For $p = 5$, the gcd of $P$ and $Q$ modulo $5$ is $x + 2$.
For $p = 271$, the gcd of $P$ and $Q$ modulo $271$ is $x + 130$.
Experimentation shows that for other values of $p$, $P$ and $Q$ are coprime modulo $p$.
\end{example}

\section{Further tools}
\label{sec:lemma}

We prove some more tools needed to establish the stability result in the next section.

\begin{lemma}
\label{lem:rays}
Let $P_1(\ul x),\ldots,P_s(\ul x) \in \Zz[\ul x]$ be nonzero coprime polynomials in $r\ge2$ variables.
Suppose that $P_i(\ul 0)\neq0$ for at least one $i\in\{1,\ldots,s\}$.
Then the polynomials $P_1(t \ul a),\ldots,P_s(t \ul a)$ are coprime in $\Qq[\ul a,t]$.
Consequently there is a proper Zariski-closed subset $Z\subset \Zz^r$ such that, for all $\ul a^\star \in \Zz^r \setminus Z$, the polynomials $P_1(t\ul a^\star),\ldots,P_s(t\ul a^\star)$ are coprime in $\Qq[t]$.
\end{lemma}

This is false if $P_1,\ldots, P_s$ vanish simultaneously at $\ul 0$.
For instance, with $P(x,y)=x$ and $Q(x,y)=y$, then $P(at,bt)= at$ and $Q(at,bt)=bt$ are not coprime, for any $(a,b) \in \Zz^2$.

\begin{corollary}
\label{cor:zariski}
Let $P_1(\ul x),\ldots,P_s(\ul x)$ be $s\geq 2$ nonzero coprime polynomials.
If $d_{\ul n_0} = 1$ for some $\ul n_0 \in \Zz^r$, then
$d_{\ul n} = 1$ for every $\ul n$ in a Zariski-dense subset of $\Zz^r$.
\end{corollary}

\begin{proof}[Proof of Corollary \ref{cor:zariski}]
With no loss of generality, assume that $\ul n_0 = \ul 0$.
By Lemma \ref{lem:rays}, for all directions $\ul a^\star$ in a Zariski-open set of $\Qq^r$, the one variable polynomials $P_1(t \ul a^\star),\ldots,P_s(t \ul a^\star)$ are coprime.
From Theorem \ref{th:stableone}, for each of these $\ul a^\star$, we have $\gcd_i P_i(k\ul a^\star) =\gcd_i P_i(\ul 0) = 1$ for all $k$ in some nonzero ideal $\delta \Zz \subset \Zz$. The set of all such $k\ul a^\star \in \Zz^r$, with varying $k$ and $a^\star$, form a Zariski-dense subset of $\Zz^r$.
\end{proof}

\begin{proof}[Proof of Lemma \ref{lem:rays}] We prove the first part; the second part easily follows by combining it with Proposition \ref{prop:bertini-intro}.
On the contrary, suppose that $P_i(t\ul a) = D(\ul a,t) \cdot P_i'(\ul a,t)$,  ($i=1,\ldots,s$) with $\deg D>0$. If $\deg_t(D)=0$, then setting $t=1$ leads to a factorization $P_i(\ul a) = D(\ul a,1) \cdot P_i'(\ul a,1)$ where $\deg D(\ul a,1) > 0$; changing the variable $\ul a$ to $\ul x$ proves that the polynomials $P_i(\ul x)$ are not coprime.

Suppose next that $\deg_t D(\ul a,t)>0$. One may assume that $\deg_t D(a_1^\star,a_2,\ldots,a_r,t)>0$
 for some $a_1^\star \in k$. For simplicity take $a_1^\star=1$ (the general case only introduces some
 technicalities).
Set $\ul a' = (1,a_2,\ldots,a_r)$ and write the decomposition in $\Qq[\ul a',t]$:
$$P_i(t\ul a') = D(\ul a',t) \cdot P_i'(\ul a',t) \quad (i=1,\ldots,s)$$
with $\deg D(\ul a',t)>0$.

Set $\ul x = t\ul a'$, that is, $x_i = a_i t$ (and $x_1 = t$); hence $a_i = x_i/x_1$ (and $a_1 = 1$), $i=1,\ldots,r$. 
Using the change of variables $(\ul a',t) \mapsto \ul x$, we obtain:
$$P_i(\ul x) = D\left(\frac{\ul x}{x_1},x_1\right) \cdot P_i'\left(\frac{\ul x}{x_1},x_1\right) \quad (i=1,\ldots,s).$$
	
By hypothesis we have $P_{i_0}(\ul 0)\neq 0$ for some $i_0\in \{1,\ldots,s\}$.
This is equivalent to $t \!\not| P_{i_0}(t\ul a')$ and implies $t \!\not| D(\ul a',t)$ in $\Qq[\ul a',t]$.

Write 
$$D(\ul a',t) = \sum_{\ul i,j} \alpha_{\ul i,j}\ul a'^{\ul i}t^j \qquad \text{ in } \Qq[a_2,\ldots,a_r,t].$$
As $a_1=1$, the multi-index $\ul i$ stands for $(0,i_2,\ldots,i_r)$ and $|\ul i| = i_2+\cdots+i_r$.
Then
\begin{align*}
D\left(\frac{\ul x}{x_1},x_1\right) 
    &= \sum_{\ul i,j} \alpha_{\ul i,j} \left( \frac{\ul x^{\ul i}}{x_1^{|\ul i|}} \right) x_1^j 
	= \sum_{\ul i,j} \alpha_{\ul i,j} \ul x^{\ul i} x_1^{j-|\ul i|} \\
	&= \frac{1}{x_1^d} \sum_{\ul i,j} \alpha_{\ul i,j} \ul x^{\ul i} x_1^{j-|\ul i|+d} 
	= \frac{1}{x_1^d} \tilde D(\ul x) \\
\end{align*}
where $d\in \Zz$, and $\tilde D(\ul x) \in \Qq[\ul x]$ is not divisible by $x_1$.  A similar computation yields $P_i'\left(\frac{\ul x}{x_1},x_1\right)=\frac{1}{x_1^{d_i}}\tilde P_i'(\ul x)$ with $d_i\in \Zz$, and $\tilde P_i'(\ul x) \in \Qq[\ul x]$ not divisible by $x_1$. This gives:
$$x_1^{d+d_i} P_i(\ul x) =  \tilde D(\ul x) \tilde P_i'(\ul x)\quad (i=1,\ldots,s).$$
By definition $\tilde D(\ul x)$ is not a monomial in $x_1$.
Moreover $\tilde D(\ul x)$ is a nonconstant polynomial. Assume on the contrary that 
$\tilde D(\ul x)$ is constant. Then $\alpha_{\underline i,j} = 0$ for $(\underline i,j) \neq (\underline 0,d)$. This implies $D(\underline a',t) = \alpha_{\underline 0,d} t^d$, in contradiction with $t \!\not| D(\ul a',t)$ and $\deg D(\underline a',t) > 0$.
Conclusion: $\tilde D(\ul x)$ is a non trivial factor of each of the $P_i(\ul x)$, hence $P_1(\ul x),\ldots,P_s(\ul x)$ are not coprime.
\end{proof}

We end by a generalization of Lemma \ref{lem:rays}.
Let $P_1(\ul x),\ldots,P_s(\ul x) \in \Zz[\ul x]$ be a family of coprime polynomials in two or more variables ($r\ge2$).

\begin{lemma}
\label{lem:second}
Let $P_1(\ul x),\ldots,P_s(\ul x) \in \Zz[\ul x]$ be nonzero coprime polynomials in $r\geq 2$ variables.
Let $\ul n \in \Zz^r$ such that $P_i(\ul n)\neq0$ for at least one $i\in\{1,\ldots,s\}$.
Then the polynomials $P_1(u\ul n + t \ul a),\ldots, P_s(u\ul n + t \ul a)$ are coprime in $\Qq[\ul a,u,t]$.
Consequently there is a proper Zariski-closed set $Z\subset \Zz^r$ such that for all $\ul a^\star \in \Zz^r \setminus Z$, 
the polynomials $P_1(u\ul n + t \ul a^\star),\ldots,P_s(u\ul n + t \ul a^\star)$ are coprime in  $\Qq[u,t]$.
\end{lemma}

\begin{proof}
For every $u^\star \in \Qq$, the polynomials $\tilde P_i(\ul x) := P_i(u^\star \ul n + \ul x)$, $i=1,\ldots,s$, are coprime in $\Qq[\ul x]$
(they are deduced from the $P_i(\ul x)$ by a mere translation on the variables).
As $P_i(\ul n)\neq 0$ for some $i$, then $\tilde P_i(\ul 0) = P_i(u^\star \ul n) \neq 0$ for all but finitely many $u^\star \in \Qq$.
By Lemma \ref{lem:rays}, for such $u^\star$, the polynomials $\tilde P_1(t \ul a),\ldots,\tilde P_s(t \ul a)$ are coprime in $\Qq[\ul a,t]$, hence 
so are the polynomials 
$P_1(u^\star \ul n + t\ul a),\ldots,P_s(u^\star \ul n + t\ul a)$.
It follows from Proposition \ref{prop:bertini-intro}
that the polynomials $P_1(u \ul n + t \ul a),\ldots,P_s(u \ul n + t \ul a)$ are coprime in $\Qq(u)[\ul a,t]$.

Assume next that their gcd in $\Qq[u,\ul a,t]$ is 
a nonconstant polynomial $D(u)\in \Qq[u]$. Thus we have $P_i(u \ul n + t \ul a)=D(u)P_i^\prime(\ul a,u,t)$ for some $P_i^\prime \in \Qq[u,\ul a,t]$, $i=1,\ldots,s$.
Choose $t^\star=1$ and $\ul a^\star(u) = -u\ul n + \ul c$, where $\ul c$ is a constant such that $P_i(\ul c) \neq 0$, for at least one $i\in \{1,\ldots,s\}$. For this choice, we have $P_i(u \ul n + t^\star \ul a^\star(u)) = P_i(\ul c)=D(u)P_i'(\ul a^\star(u),u,t^\star)$. As $P_i(\ul c)$ is a nonzero constant, $D(u)$ is a constant polynomial. 

By Remark \ref{rk:bertini}(a), the polynomials $P_1(u \ul n + t \ul a),\ldots, P_s(u \ul n + t \ul a)$ are coprime in $\Qq[u][\ul a,t]$.

This proves the first assertion of Lemma \ref{lem:second}; the second one follows by combining it with Proposition \ref{prop:bertini-intro}.
\end{proof}

\section{Proof of the stability}
\label{sec:proof}

This section is devoted to the proof of Theorem \ref{th:main}.
\smallskip

\textbf{Idea of the proof.} Consider two coprime polynomials $P(x,y)$ and $Q(x,y)$
and the special case of two pairs $(m,n_1)$ and $(m,n_2)$ (with the same $x$-coordinate).
We will find $n_3$ such that $\gcd(d_{m,n_1},d_{m,n_2}) = d_{m,n_3}$.
As $P(x,y)$ and $Q(x,y)$ are coprime and by Bézout, there exist $A(x), B(x), R(x) \in \Zz[x]$ such that:
$$A(x)P(x,y) + B(x)Q(x,y) = R(x).$$

For all $m\in\Zz$ but finitely many, we have $R(m) \neq 0$. 
For such $m$, $P(m,y)$ and $Q(m,y)$ are coprime (in $\Qq[y]$).
By the gcd stability result in one variable (Theorem \ref{th:stableone}), 
there exists $n_3$ such that $\gcd(d_{m,n_1},d_{m,n_2}) = d_{m,n_3}$.
\smallskip

The proof extends this idea: we need (a) to deal with the case where $P(m,y)$ and $Q(m,y)$ are no longer coprime; (b) also consider pairs $(m_1,n_1)$ and $(m_2,n_2)$ with $m_1\neq m_2$.

\medskip
\textbf{Step 1.}
Let $\ul m\in\Zz^r$ and $\ul n\in\Zz^r$.
For simplicity, and with no loss of generality, we assume $\ul m= \ul 0$.
We may also assume that $P_i(\ul 0)\not=0$ for at least one $i\in\{1,\ldots,s\}$: otherwise $d_{\ul 0}=0$ so that we can directly conclude $\gcd(d_{\ul 0},d_{\ul n}) = d_{\ul n}$. We may also assume that $P_i(\ul n)\not=0$ for at least one $i\in\{1,\ldots,s\}$.
We reduce from several to one variables by restricting the polynomials on the line passing through $\ul 0$ and $\ul n$.
That is, we set:
$$P^0_i(t) = P_i(t\ul n), \qquad i=1,\ldots,s.$$
Then $P^0_i(0)=P_i(\ul 0)$ and $P^0_i(1)=P_i(\ul n)$.
However the polynomials $P^0_1(t),\ldots, P^0_s(t)$ are not necessarily coprime.
The following picture helps visualize the next steps of the proof.

\bigskip
\textbf{Picture of the proof.}

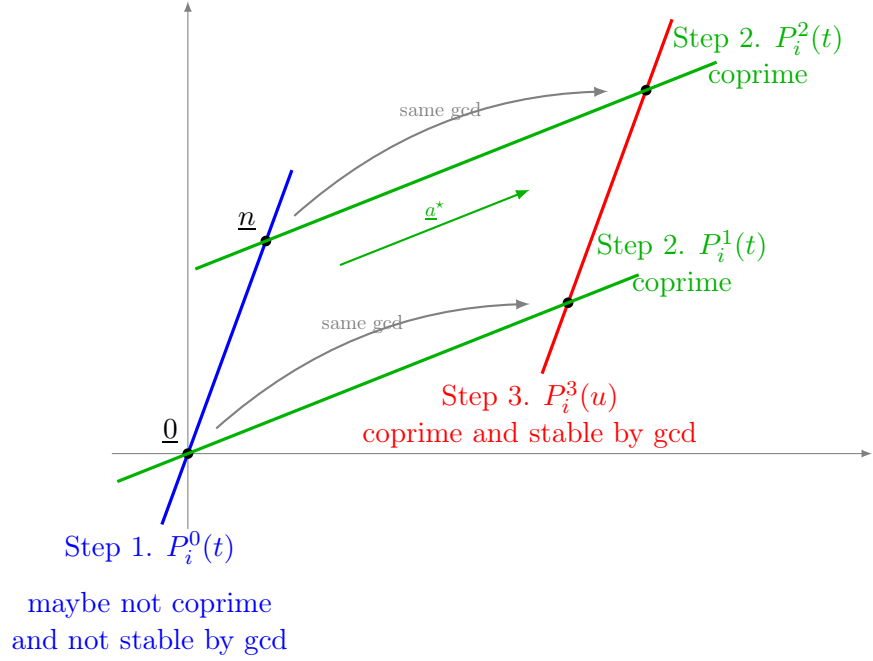
\begin{center}
\begin{tikzpicture}

\draw[->,>=latex,gray] (-1,0) -- (9,0);
\draw[->,>=latex,gray] (0,-1) -- (0,6);

\coordinate (A) at (0,0);
\coordinate (B) at (70:3);
\node  at (A)[above left] {$\underline 0$};
\node at (B) [above left] {$\underline n$};
\draw[very thick, blue,shorten >=-1cm,shorten <=-1cm] (A) -- (B);
\fill (A) circle (2pt);
\fill (B) circle (2pt);

\coordinate (AA) at ($(5,2)+(0,0)$);
\coordinate (BB) at  ($(5,2)+(70:3)$);
\draw[very thick, red,shorten >=-1cm,shorten <=-1cm] (AA) -- (BB);
\fill (AA) circle (2pt);
\fill (BB) circle (2pt);

\draw[very thick, green!70!black,shorten >=-1cm,shorten <=-1cm] (A) -- (AA);
\draw[very thick, green!70!black,shorten >=-1cm,shorten <=-1cm] (B) -- (BB);

\node[blue] at (-0.5,-1.25){Step 1. $P_i^0(t)$};
\node[blue, text width=4cm, align=center] at (-0.5,-2.25){maybe not coprime and not stable by gcd}; 

\node[red] at (4.5,0.75){Step 3. $P_i^3(u)$};
\node[red] at (4.5,0.25){coprime and stable by gcd}; 

\node[green!70!black] at (6.5,2.75){Step 2. $P_i^1(t)$};
\node[green!70!black] at (6.5,2.25){coprime};

\node[green!70!black] at (7.5,5.5){Step 2. $P_i^2(t)$};
\node[green!70!black] at (7.5,5.0){coprime};

\draw[->,>=latex,gray,thick,shorten >=0.5cm,shorten <=0.5cm] (A) to[bend left=20] node[midway,above,scale=0.7]{same gcd}(AA);
\draw[->,>=latex,gray,thick,shorten >=0.5cm,shorten <=0.5cm] (B) to[bend left=20] node[midway,above,scale=0.7]{same gcd}(BB);

\draw[->,>=latex,green!70!black,thick] (2,2.5) -- ++(2.5,1) node[midway,above,scale=0.7] {$\underline a^\star$};

\end{tikzpicture}

\end{center}

\medskip
\textbf{Step 2.}
By Lemma \ref{lem:rays}, for all $\ul a^\star\in \Zz^r$ but in a proper Zariski-closed set,
the polynomials $P_1(t\ul a^\star),\ldots, P_s(t\ul a^\star)$ are coprime in $\Qq[t]$.
Moreover, again by Lemma \ref{lem:rays} centered at $\ul n$, for all $\ul a^\star\in \Zz^r$ but in a proper Zariski-closed set, the polynomials $P_1(\ul n + t\ul a^\star), \ldots, P_s(\ul n + t\ul a^\star)$ are coprime in $\Qq[t]$.
Finally by Lemma \ref{lem:second}, for all $\ul a^\star\in \Zz^r$ but in a proper Zariski-closed set, the polynomials $P_1(u\ul n + t \ul a^\star), \ldots,P_s(u\ul n + t \ul a^\star)$ are coprime in $\Qq[u,t]$.

Pick $\ul a^\star\in \Zz^r$ such that the following conditions are satisfied:
\begin{itemize}
  \item $P^1_i(t):=P_i(t\ul a^\star)$, $i=1,\ldots,s$, are coprime in $\Qq[t]$,
  \item $P^2_i(t):=P_i(\ul n + t\ul a^\star)$, $i=1,\ldots,s$, are coprime in $\Qq[t]$,
  \item $P_i(u\ul n + t \ul a^\star)$, $i=1,\ldots,s$, are coprime in $\Qq[u,t]$.
\end{itemize} 

In the computations below, all gcds are computed with respect to the indices $i=1,\ldots,s$.

By the one variable case for $P^1_1(t),\ldots,P^1_s(t)$, the corresponding sequence of gcd is periodic, for some (nonzero) period $\delta_1\in \Zz$ (Theorem \ref{th:stableone}).  
This yields that for any $k\in\Zz$, we have $\gcd P^1_i(0)= \gcd P^1_i (0+k\delta_1)$,
and so
$$d_{\ul 0} = \gcd P_i(\ul 0) = \gcd P_i(k\delta_1\ul a^\star).$$

We do the same for $P^2_i(t)$. For some period $\delta_2\in \Zz$, for any $k\in\Zz$, we have
$\gcd P^2_i(0) = \gcd P^2_i(0+k\delta_2)$, and so
$$d_{\ul n} = \gcd P_i(\ul n) = \gcd P_i(\ul n + k\delta_2 \ul a^\star).$$

We also have $P_1(u\ul n + t \ul a^\star), \ldots,P_s(u\ul n + t \ul a^\star)$ coprime in $\Qq[u,t]$. Thus, by 
Proposition \ref{prop:bertini-intro},
for all but finitely many $t^\star\in \Qq$, the polynomials $P_1(u\ul n + t^\star \ul a^\star),\ldots,
P_s(u\ul n + t^\star \ul a^\star)$ are coprime in $\Qq[u]$.
 
\medskip
\textbf{Step 3.}
Set $t^\star= k \delta_1\delta_2$ with $k\in \Zz$ and $P^3_i(u) := P_i(u\ul n + t^\star\ul a^\star)$, $i=1,\ldots,s$.
Pick $k$ large enough to guarantee that $P^3_1(u),\ldots, P^3_s(u)$
 are coprime in $\Qq[u]$ (Proposition \ref{prop:bertini-intro}).

Note that 
$$\gcd P^3_i(0) = \gcd P_i(t^\star\ul a^\star) = \gcd P_i(k\delta_1\delta_2\ul a^\star) = \gcd P_i(\ul 0) = d_{\ul 0}$$
and
$$\gcd P^3_i(1) = \gcd P_i(\ul n + t^\star\ul a^\star) = \gcd P_i(\ul n + k\delta_1\delta_2\ul a^\star) =  \gcd P_i(\ul n) = d_{\ul n}.$$

Now by the gcd stability (resp.{} lcm stability) assertion from Theorem \ref{th:stableone}, applied to 
the one variable coprime polynomials $P^3_1(u),\ldots,P^3_s(u)$, 
there exists
$\ell\in\Zz$ such that 
$$\gcd\left( \gcd P^3_i(0), \gcd P^3_i(1) \right) = \gcd P^3_i(\ell)$$
(resp.{} $\lcm\left( \gcd P^3_i(0), \gcd P^3_i(1) \right) = \gcd P^3_i(\ell)$).
Setting $\ul m = \ell \ul n + t^\star \ul a^\star$, so $P^3_i(\ell) = P_i(\ul m)$, 
we obtain
$$\gcd\left( d_{\ul0}, d_{\ul n} \right) = d_{\ul m}$$
(resp.{} $\lcm\left( d_{\ul0}, d_{\ul n} \right) = d_{\ul m}$),
which proves the stability of $\mathcal{D}$ by gcd (resp.{} lcm).

\section{Proof of the Ekedahl--Poonen formula}
\label{sec:proof-poonen}

This section is mainly devoted to the proof of the Ekedahl--Poonen formula
as stated in Theorem \ref{th:poonen}. 
While \cite[Theorem 3.1]{Poo} is valid over the rings $\Zz$ and $\Ff_q[t]$, here we state and prove Theorem \ref{th:poonen} over $\Zz$ only, which enables simplifications. Another simplification is that our density is defined by squared boxes, while  \cite{Poo} allows rectangular ones. Another difference 
(minor for the proof, but important for the applications) is that we allow any $s\ge2$ polynomials (instead of $2$). Finally in Section \ref{ssec:gen_sev_pol}, we generalize the formula to the situation of \emph{several families}
of coprime polynomials (Proposition \ref{prop:poonen-gen}), and then use this generalization to extend the proof of Theorem \ref{cor:HS} given in Section \ref{ssec:HS-intro} for one polynomial to several polynomials.

\subsection{Sets}
As usual, fix $s\geq 2$ nonzero polynomials $P_1(\ul x), \ldots, P_s(\ul x) \in \Zz[\ul x]$.
In the following,  $p$ is a prime number, and $\mathcal{P}$ the set of prime numbers.

For $p \in \mathcal{P}$, consider the set:
$$\mathcal{R}_p = \big\{ \ul n \in \Zz^r \mid p \text{ does not divide all } P_1(\ul n),\ldots,P_s(\ul n) \big\}.$$

Then, with $\textstyle \mathcal{R}$ the set (introduced in Section \ref{ssec:poonen-intro}) of all $\ul n \in \Zz^r $ such that $P_1(\ul n), \ldots,P_s(\ul n)$ are coprime,
we have:

$$\mathcal{R} = \bigcap_{p \in \mathcal{P}} \mathcal{R}_p = \left\{ \ul n \in \Zz^r \mid \gcd (P_1(\ul n),\ldots,P_s(\ul n)) = 1 \right\}.$$

We will approximate $\mathcal{R}$ by sets $\mathcal{R}_{\le M}$ defined by:
$$\mathcal{R}_{\le M} = \bigcap_{p \le M} \mathcal{R}_p = \big\{ \ul n \in \Zz^r \mid \text{ for every } p \le M, p \text{ does not divide all } P_1(\ul n),\ldots,P_s(\ul n)  \big\}.$$

We will also  work with:
$$\mathcal{Q}_p 
= \Zz^r \setminus \mathcal{R}_p 
= \big\{ \ul n \in \Zz^r \mid p \text{ divides } P_1(\ul n),\ldots,P_s(\ul n) \big\}
= \big\{ \ul n \in \Zz^r \mid p \text{ divides } \gcd_{1\leq i\leq s}P_i(\ul n) \big\}.$$

and 
\begin{align*}
\mathcal{Q}_{> M} 
&= \bigcup_{p > M} \mathcal{Q}_p 
= \big\{ \ul n \in \Zz^r \mid \text{ there exists } p >M, p \text{ divides } P_1(\ul n),\ldots,P_s(\ul n) \big\} \\
&= \big\{ \ul n \in \Zz^r \mid \text{ there exists } p >M \text{ such that } p \text{ divides } \gcd_{1\leq i\leq s}P_i(\ul n) \big\}. 
\end{align*}

Here are the main steps of the proof:
\begin{itemize}
  \item Compute the density of $\mathcal{Q}_p$ (and $\mathcal{R}_p$) in terms of $c_p$.

  \item Prove that this density is in $O(\frac1{p^2})$.

  \item Compute the density of $\mathcal{R}_{\le M}$ from $\mathcal{R}_p$, using the Chinese Remainder Theorem.

  \item Prove that 
$\mu(\mathcal{R}_{\le M}) \xrightarrow[M\to+\infty]{} \mu(\mathcal{R})$.
\end{itemize}

For $r=1$, the last step is not necessary since, following notation of Section \ref{ssec:poonen-one}, for $M \ge \delta$, we have $\mathcal{R}_{\le M} = \mathcal{R}$.

\subsection{Density of $\mathcal{Q}_{p}$ and $\mathcal{R}_{p}$}

By definition,  $\ul n \in \mathcal{Q}_{p}$ if and only if 
$P_i(\ul n) = 0 \pmod p$ for each $i=1,\ldots,s$.
Hence 
\begin{equation}
\label{eq:cp}
\# (\mathcal{Q}_{p} \cap \llbracket 0,p-1 \rrbracket^r) = c_p
\end{equation}

In fact, $p$ divides $P_i(n_1,\ldots,n_r)$ if and only if $p$ divides $P_i(n_1+k_1p,\ldots,n_r+k_rp)$ for any $k_j\in\Zz$.
Hence $\mathcal{Q}_{p}$ is invariant by any translation of vector $(k_1p,\ldots,k_rp)$ (with $k_j\in \Zz$).
Hence, as a function of $B$, the cardinality $\# (\mathcal{Q}_{p} \cap \mathbb{B})$ (with $\mathbb{B} =  \llbracket 0,B-1 \rrbracket^r$) is asymptotic to $c_p \left( \frac B p \right)^r$ as $B\rightarrow \infty$ (this formula is exact if $p$ divides $B$).

Then:
\begin{equation}
\label{eq:qp}
\mu(\mathcal{Q}_p) = \lim_{B \to +\infty} \frac{\# (\mathcal{Q}_{p} \cap \mathbb{B}) }{\# \mathbb{B}}
= \frac{c_p}{p^r}
\end{equation}
As $\mathcal{R}_p = \Zz^r \setminus \mathcal{Q}_p$ we also get:
\begin{equation}
\label{eq:rp}
\mu(\mathcal{R}_p) = 1 - \frac{c_p}{p^r}
\end{equation}

\subsection{Bound for $\mathcal{Q}_{p}$}

We need to bound the number $c_p$ of solutions in $(\Zz/p\Zz)^r$ of the set of equations 
$P_i(\ul n) = 0 \pmod p$ ($i=1,\ldots,s$).
If $r=1$, we explained in Section \ref{ssec:poonen-intro} that $c_p=0$ for all suitably large primes $p$.
For $r\geq 2$, one can bound $c_p$ using the Bézout theorem over $\Zz/p\Zz$. For $r=2$, one can use for instance \cite[Theorem 4.1]{Tao}. For $r\ge2$, we have this general version, by Lachaud--Rolland \cite[Corollary 2.2]{LR}:
\smallskip

\textbf{General Bézout theorem.} 
\emph{Let $r\ge2$. We have $c_p \le d^s \cdot p^m$, where $m$ is the dimension of the zero-set of the polynomials $P_1(\ul x),\ldots, P_s(\ul x)$, assumed to be of degree $\le d$.}

\begin{corollary}
\label{cor:bezout}
For all sufficiently large $p$, we have $c_p \le d^s \cdot p^{r-2}$. 
Consequently, we obtain $\mu(\mathcal{Q}_p) = O\left(\frac{1}{p^2}\right)$.
\end{corollary}

\begin{proof}[Proof of Corollary \ref{cor:bezout}]
The polynomials $P_1(\ul x),\ldots, P_s(\ul x)$ are coprime in $\Zz[\ul x]$. By 
Corollary \ref{cor:ostrowski}, they are coprime
in $\Qq[\ul x]$ and the polynomials $\overline{P_1}^{p}(\ul x),\ldots, \overline{P_s}^p(\ul x)$ (reduced modulo $p$) are nonzero and coprime in $\Ff_p[\ul x]$ for all suitably large primes $p$. It follows that they  are coprime in $\overline{\Ff_p}[\ul x]$ for the same primes $p$ (this is explained for example in \cite[\S 2.1]{BDNschinzel}).

Fix such a prime $p$ and consider the ideal $\mathcal{I} = \langle \overline{P_1}^{p},\ldots, \overline{P_s}^p\rangle \subset \overline{\Ff_p}[\ul x]$. We estimate below the dimension of the zero-set $Z(\mathcal{I})\subset \overline{\Ff_p}^r$ of $\mathcal{I}$ and then we will apply the general B\'ezout theorem.  
Classically this dimension is also the Krull dimension $\dim  \overline{\Ff_p}[\ul x]/\mathcal{I}$ of the quotient ring $\overline{\Ff_p}[\ul x]/\mathcal{I}$ (e.g. \cite[Proposition 1.7]{hartshorne}). 

By definition, $\dim \overline{\Ff_p}[\ul x]/\mathcal{I}$ is  the supremum of the heights of minimal prime ideals of $\overline{\Ff_p}[\ul x]$ containing $\mathcal{I}$.
We may assume that $\deg \overline{P_1}^{p} \geq 1$; otherwise $c_p=0$. Then $\overline{P_1}^{p}$ has at least one irreducible factor $\Delta \in \overline{\Ff_p}[\ul x]$. Furthermore the prime ideal $\langle \Delta \rangle \subset \overline{\Ff_p}[\ul x]$ is not maximal (by Nullstellensatz and $r\geq 2$), but is contained in a maximal ideal. We deduce that $ {\rm height}(\langle {\Delta}\rangle)\geq 1$, and, by \cite[Theorem 1.8 A]{hartshorne}, that
$$\dim {\overline{\Ff_p}[\ul x]/\mathcal{I}} \leq \dim{\overline{\Ff_p}[\ul x]/\langle \overline{P_1}^{p}\rangle} \leq r-1.$$

Assume that $\dim {\overline{\Ff_p}[\ul x]/\mathcal{I}} = r-1$. Let ${\mathfrak{p}} \subset \overline{\Ff_p}[\ul x]$ be a minimal prime ideal containing $\mathcal{I}$; thus $\dim {\overline{\Ff_p}[\ul x]/{\mathfrak{p}}} = r-1$, or, equivalently ${\mathfrak{p}}$ is of height $1$. 
By Krull's Hauptidealsatz \cite[Theorem 1.11 A \& Proposition 1.13]{hartshorne}, the variety $Z({\mathfrak{p}})$ is a hypersurface $Z(f)$, for some irreducible polynomial $f\in \overline{\Ff_p}[\ul x]$. But then it follows from $ \langle f \rangle = {\mathfrak{p}} \supset \mathcal{I}$ that $f$ divides each polynomial $\overline{P_i}^p$ in $\overline{\Ff_p}[\ul x]$, $i=1,\ldots,s$, a contradiction. Conclude that $\dim {\overline{\Ff_p}[\ul x]/\mathcal{I}} \leq r-2$.

The first assertion of Corollary \ref{cor:bezout} then readily follows from the General B\'ezout theorem, and the second one from this easy estimate:
$$\mu(\mathcal{Q}_p) = \frac{c_p}{p^r} \le  \frac{d^s \cdot p^{r-2}}{p^r} = \frac{d^s}{p^2} =  O\left(\frac{1}{p^2}\right).$$
\end{proof}

\subsection{The set $\mathcal{R}_{\le M}$}

Let $M\ge0$, let $\{p_1,\ldots,p_\ell\}$ be the set of primes $\le M$ and $N$ be the product of these primes.

The Chinese Remainder Theorem gives an isomorphism from
$\Zz/N\Zz$ to $\Zz/p_1\Zz \times \cdots \times \Zz/p_\ell\Zz$,
which we extend to the dimension $r$ by

\centerline{$\ul {n} \in (\Zz/N\Zz)^r \mapsto (\ul{n_1},\ldots,\ul{n_\ell}) \in (\Zz/p_1\Zz)^r \times \cdots \times (\Zz/p_\ell\Zz)^r$,}

where $\ul{n_j}$ is $\ul {n}$ modulo $p_j$.
We have a 1-1 correspondence between the sets 
$\mathcal{R}_{\le M}$ and $\mathcal{R}_{p_1}\times \cdots \times \mathcal{R}_{p_\ell}$. 
Namely:
\begin{align*}
     & \ul{n} \in \mathcal{R}_{\le M} \cap \llbracket 0,N-1 \rrbracket^r \\
\iff\quad & \forall j \in \{1,\ldots,\ell\} \quad \exists i \in \{1,\ldots,s\} \qquad P_i(\ul n) \neq 0 \pmod {p_j} \\
\iff\quad & \forall j \in \{1,\ldots,\ell\} \quad \exists i \in \{1,\ldots,s\} \qquad P_i(\ul {n_j}) \neq 0 \pmod {p_j} \\ 
\iff\quad &\forall j \in \{1,\ldots,\ell\}  \quad \ul {n_j}\in \mathcal{R}_{p_j} \cap \llbracket 0,p_j-1 \rrbracket^r. \\
\end{align*}

Recall that $\mathcal{R}_p = \Zz^r \setminus \mathcal{Q}_p$. Thus, with \eqref{eq:cp}, we obtain:
$$\# (\mathcal{R}_{p} \cap \llbracket 0,p-1 \rrbracket^r)
= p^r - \# (\mathcal{Q}_{p} \cap \llbracket 0,p-1 \rrbracket^r)
= p^r - c_p.$$
Whence:
$$\# (\mathcal{R}_{\le M} \cap \llbracket 0,N-1 \rrbracket^r)
= \prod_{j=1}^\ell (p_j^r - c_{p_j}).$$

This provides the density of $\mathcal{R}_{\le M}$:
$$\mu(\mathcal{R}_{\le M}) 
= \lim_{B \to + \infty} \frac{\# (\mathcal{R}_{\le M} \cap \mathbb{B})}{\# \mathbb{B}}
= \lim_{B \to + \infty} \frac{\left(\frac{B}{N}\right)^r \prod_{j=1}^\ell (p_j^r - c_{p_j})}{B^r}
= \prod_{j=1}^\ell \left(1 - \frac{c_{p_j}}{p_j^r}\right).$$
Whence:
\begin{equation}
\label{eq:densityrm}
\mu(\mathcal{R}_{\le M}) = \prod_{p \le M} \left(1 - \frac{c_{p}}{p^r}\right)
\end{equation}

\subsection{Limit of $\mathcal{Q}_{>M}$}

\begin{lemma} We have:
	\begin{equation}
		\label{eq:limitqm}
		\mu(\mathcal{Q}_{>M}) \xrightarrow[M \to +\infty]{} 0
	\end{equation}
\end{lemma}

The proof (here, for several polynomials) is similar to \cite[Lemma 5.1]{Poo} (for two polynomials) with some simplifications.

\begin{proof}
	
	Fix $M$ and $B \ge M$ and consider the decomposition:
	$$\mathcal{Q}_{>M} = \mathcal{Q}_{>M, \le B} \cup \mathcal{Q}_{>B},$$
	where $\mathcal{Q}_{>M, \le B} = \bigcup_{M < p \le B} \mathcal{Q}_p$ and
	$\mathcal{Q}_{>B} = \bigcup_{p > B} \mathcal{Q}_p$.
	We will prove that each term has a relatively small cardinal compared to $B^r = \# \mathbb{B}$, where $\mathbb{B} = \llbracket 0,B-1 \rrbracket^r$. 
	
	\medskip
	
	\textbf{Estimate of $\mathcal{Q}_{>M, \le B}$.}
	
	From Corollary \ref{cor:bezout},  
	we have $c_p = \# (\mathcal{Q}_{p} \cap \llbracket 0, p-1 \rrbracket^r) \le d^s \cdot p^{r-2}$, which gives $\# (\mathcal{Q}_{p} \cap \mathbb{B}) \le C p^{r-2} \left(\frac{B}{p}\right)^r$
	for some constant $C$ (depending only on $d$ and $s$). Thus we obtain:
	\begin{equation}
		\label{eq:QMB}
		\frac{\# (\mathcal{Q}_{>M} \cap \mathbb{B})}{\# \mathbb{B}} 
		\le \sum_{M < p \le B} \frac{\# (\mathcal{Q}_{p} \cap \mathbb{B})}{\# \mathbb{B}}  
		\le C \sum_{M < p \le B} \frac{1}{p^2} 
		\le C \sum_{p > M} \frac{1}{p^2} 
	\end{equation}
	The last term does not depend on $B$ and tends to $0$ as $M \to +\infty$.

	\medskip
	
	\textbf{Estimate of $\mathcal{Q}_{>B}$.}
	
	\emph{Preliminaries.}
	Firstly, we may reduce to the case where each polynomial $P_i$ is irreducible in $\Zz[\ul x]$. Indeed assume $P_1 = Q \cdot R$ with $Q,R \in \Zz[\ul x]$. If $p | P_1(\ul n)$ then $p | Q(\ul n)$ or $p | R(\ul n)$. Hence $\mathcal{Q}_{>B}(QR,P_2,\ldots,P_s) \subset \mathcal{Q}_{>B}(Q,P_2,\ldots,P_s) \cup \mathcal{Q}_{>B}(R,P_2,\ldots,P_s)$.
	
	Secondly, we may also reduce to the case where one of the polynomials, say $P_s$, is a polynomial in $x_1,\ldots,x_{r-1}$ only. Namely, as $P_1,\ldots,P_s$ are coprime, we have a Bézout identity $\sum_{i=1}^s A_i(\ul x)P_i(\ul x) = \Delta(x_1,\ldots,x_{r-1})$
	with $A_i \in \Zz[\ul x]$, $i=1,\ldots,s$ and $\Delta \in \Zz[x_1,\ldots,x_{r-1}]$, $\Delta \neq 0$.
	If $p | P_i(\ul n)$ for $i=1,\ldots,s$ then $p | \Delta(\ul n)$. Hence
	$\mathcal{Q}_{>B}(P_1,\ldots,P_s) \subset \mathcal{Q}_{>B}(P_1,\ldots,P_{s-1},\Delta)$
	and $P_1,\ldots,P_{s-1},\Delta$ are coprime polynomials (if some nonconstant polynomial $R$ divides $P_1$, a polynomial where all the variables $x_1,\ldots, x_r$ occur, then $R=P_1$, up to some multiplicative constant, because $P_1$ is irreducible, but then $R$ cannot divide $\Delta$ in the variables $x_1,\ldots,x_{r-1}$ only). 
	
	Thirdly, we may assume that the leading coefficient of $P_i(\ul x)$, $i=1,\ldots,s-1$, seen as a polynomial in $x_r$, is not divisible by the last polynomial $P_s(x_1,\ldots,x_{r-1})$.
	Indeed, write $P_i(\ul x) = P_i^{0}(x_1,\ldots,x_{r-1}) x_r^{\delta_i} + \cdots \in \Zz[x_1,\ldots,x_{r-1}][x_r]$. If $P_i^0 = Q_i P_s$, then $P_i' = P_i - Q_i P_s x_r^{\delta_i}$ is a polynomial with $\deg_{x_r}(P_i') < \deg_{x_r}(P_i)$. We proceed by induction on $\deg_{x_r}(P_i)$ until $P_s$ does not divide $P_i^0$ (or $\deg_{x_r}(P_i)=0$). Note that the set $\mathcal{Q}_p$ is preserved in this process:
	$p | P_i(\ul n)$ and $P_s(\ul n)$ iff $p$ divides $(P_i - Q_i P_s x_r^{\delta_i})(\ul n)$ and $P_s(\ul n)$. It may then be necessary to apply the first reduction to this new set of polynomials, to get irreducible polynomials.

	We prove below that 
	\begin{equation}
		\label{eq:QB}	
		\frac{\# (\mathcal{Q}_{>B} \cap \mathbb{B})}{\# \mathbb{B}} \xrightarrow[B \to +\infty]{} 0
	\end{equation}
	
	\emph{Induction.} The proof is by induction on the dimension $r$.
	For $r=1$, a Bézout identity $\sum_{i=1}^s A_i(x)P_i(x) = \Delta$ (with $A_i \in \Zz[x]$, $\Delta \in \Zz$, $\Delta \neq 0$) implies that if $p | P_i(n)$ for every $i=1,\ldots,s$, then $p | \Delta \in \Zz$. Hence for $B>\Delta$, $\mathcal{Q}_{>B} = \varnothing$.
	
	For $r \ge 1$, we introduce the three following subsets $\mathcal{S}_1$, $\mathcal{S}_2$, $\mathcal{S}_3$ and work with the inclusion:
	$\mathcal{Q}_{>B} \cap \mathbb{B} \subset \mathcal{S}_1 \cup \mathcal{S}_2 \cup \mathcal{S}_3$.
	\begin{itemize}
		\item $\mathcal{S}_1 = \{ \ul n \in \mathbb{B} \mid P_s(\ul n) = 0 \}$. By the Zippel--Schwartz lemma, $\#\mathcal{S}_1/\#\mathbb{B}$ tends to $0$ as $B \to +\infty$.
		
		\item $\mathcal{S}_2 = \{ \ul n \in \mathbb{B} \mid \exists p > B, p|P_1^0(\ul n),\ldots,p | P_{s-1}^0(\ul n), p | P_s(\ul n)  \}$. By induction, $\#\mathcal{S}_2/\#\mathbb{B}$ tends to $0$ as $B \to +\infty$.
		
		\item $\mathcal{S}_3 = \{ \ul n \in \mathbb{B} \mid P_s(\ul n) \neq 0, \exists p > B,  p|P_1(\ul n),\ldots, p | P_s(\ul n),  p \nmid P_{i_0}^0(\ul n) \text{ for some } 1 \le i_0 < s \}$.
		For all $\ul n \in \mathbb{B}$, $P_s(\ul n) = O(B^{\gamma})$ with $\gamma = \deg(P_s)$. 
		Fix $(n_1,\ldots,n_{r-1}) \in \llbracket 0,B-1 \rrbracket^{r-1}$ and consider a $r$-tuple $\ul n= (n_1,\ldots,n_{r-1},n_r)$ in the set $\mathcal{S}_3$. For any sufficiently large $B$, there are at most $\gamma$ possible primes $p>B$ such that $p$ divides the nonzero integer $P_s(\ul n)$.	
		Pick such a prime $p$ and let $i\in \{1,\ldots,s-1\}$ be an index such that $p \nmid P_{i}^0(n_1,\ldots,n_{r-1})$. Then write: 
		$$P_{i}(n_1,\ldots,n_{r-1},x)
		= P_{i}^0(n_1,\ldots,n_{r-1}) x^{\delta_i} + \cdots$$
		There are at most $\delta_i$ integers $x=n_r$ with $0 \le n_r < p$, hence a fortiori with $0 \le n_r < B$, such that $p | P_{i}(n_1,\ldots,n_{r-1},n_r)$.
		This shows that for each $(n_1,\ldots,n_{r-1}) \in \llbracket 0,B-1 \rrbracket^{r-1}$, there are at most $C'=(s-1) \cdot \gamma \cdot \delta_1 \cdots \delta_{s-1}$ values of $n_r \in \llbracket 0, B-1 \rrbracket$ such that $(n_1,\ldots,n_{r-1},n_r) \in \mathcal{S}_3$.
		Hence $\#\mathcal{S}_3/\#\mathbb{B} \le \frac{B^{r-1} \cdot C'}{B^r} = \frac{C'}{B}$ tends to $0$ as $B \to +\infty$.
		
		
	\end{itemize}
	
	\medskip
	
	\textbf{Conclusion.}
	As $\mathcal{Q}_{>M} = \mathcal{Q}_{>M, \le B} \cup \mathcal{Q}_{>B}$,
	then by \eqref{eq:QMB}:
	$$\frac{\#\mathcal{Q}_{>M}}{\# \mathbb{B}} \le 
	\frac{\# (\mathcal{Q}_{>M, \le B} \cap \mathbb{B})}{\# \mathbb{B}} 
	+ \frac{\# (\mathcal{Q}_{>B} \cap \mathbb{B})}{\# \mathbb{B}}
	\le C \sum_{p > M} \frac{1}{p^2} + \frac{\# (\mathcal{Q}_{>B} \cap \mathbb{B})}{\# \mathbb{B}}.$$
	By \eqref{eq:QB}, the last term, tends to $0$ as $B \to +\infty$,
	and the first term tends to $0$ as $M  \to +\infty$.
	It yields that $\mu(\mathcal{Q}_{>M}) \xrightarrow[M \to +\infty]{} 0$.

\end{proof}
		

\subsection{Limit of $\mathcal{R}_{\le M}$}

We have $\mathcal{R} \subset \mathcal{R}_{\le M}$.
Note that $\mathcal{R}_{\le M} \setminus \mathcal{R}
 \subset \mathcal{Q}_{> M}$: in fact $\mathcal{R}_{\le M} \setminus \mathcal{R}$
is the set of $\ul n$ for which the primes $p$ that divide all the $P_i(\ul n)$ verify $p>M$, such $\ul n$ are in the union of
the $\mathcal{Q}_{p}$, for $p>M$. 

Consider the decomposition:
$$\mathcal{R}_{\le M} 
= \mathcal{R} \cup (\mathcal{R}_{\le M} \setminus \mathcal{R})
\subset \mathcal{R} \cup \mathcal{Q}_{> M}.
$$ 

It yields the inequalities:
$$\mu(\mathcal{R}) \le \mu(\mathcal{R}_{\le M} ) \le \mu(\mathcal{R}) + \mu(\mathcal{Q}_{> M}).$$

As, by \eqref{eq:limitqm}, $\mu(\mathcal{Q}_{>M}) \xrightarrow[M \to +\infty]{} 0$, we obtain
\begin{equation}
\label{eq:limitrm}
\mu(\mathcal{R}_{\le M}) \xrightarrow[M \to +\infty]{} \mu(\mathcal{R})
\end{equation}

As $\mu(\mathcal{R}_{\le M}) = \prod_{p \le M} \left(1 - \frac{c_{p}}{p^r}\right)$ by \eqref{eq:densityrm},
then
$$\mu(\mathcal{R}) = \prod_{p \in \mathcal{P}} \left(1 - \frac{c_{p}}{p^r}\right).$$

This infinite product is nonzero if no prime $p$ divides all the values of $P_1(\ul n),\ldots,P_s(\ul n)$ for all $\ul n\in \Zz^r$, 
i.e., if $c_p \neq p^r$ for all primes $p$.

\subsection{Generalization to several families of polynomials} \label{ssec:gen_sev_fam} \label{ssec:gen_sev_pol}

Consider $\ell\geq 1$ families ${\mathcal P}_j=\{ P_{j1}(\ul x),\ldots,P_{js_j}(\ul x)\}$ of nonzero coprime polynomials in $\Zz[\ul x]$, $j=1,\ldots,\ell$.
For each $j=1,\ldots,\ell$, consider the set
$$\mathcal{R}({\mathcal P}_j) = \left\{ \ul n \in \Zz^r \mid \gcd (P_{j1}(\ul n),\ldots,P_{\ell s_\ell}(\ul n)) = 1 \right\}.$$

Our goal is to evaluate the set $\mathcal{R} = \bigcap_{j=1}^\ell \mathcal{R}({\mathcal P}_j)$.

\begin{proposition} \label{prop:poonen-gen} 
Let $\Pi = {\mathcal P}_1 \cdots {\mathcal P}_\ell \subset \Zz[\ul x]$ be the set of all possible
products $A_1 \cdots A_\ell$ with $A_j\in {\mathcal P}_j$ for $j=1,\ldots,\ell$. Then we have the following:
\begin{itemize}
  \item[(a)] The elements of $\Pi$ are nonzero coprime polynomials (in $\Qq[\ul x]$). 
  \item[(b)] $\mathcal{R} = \mathcal{R}(\Pi)$.
  \item[(c)] $\displaystyle \mu(\mathcal{R}) = \prod_{p \in \mathcal{P}} \left( 1 - \frac{c_p}{p^r} \right)$
where 
$c_p = \# \big\{ \ul n \in (\Zz/p\Zz)^r \mid Q(\ul n) = 0 \pmod p, \forall Q \in \Pi \big\}$.
  \item[(d)] For every $p\in {\mathcal P}$, we have $c_p=p^r$ if and only if for every $\ul n \in \Zz^r$, there exists $j\in \{1,\ldots,\ell\}$ such that the prime  $p$ divides all values $P_{j1}(\ul n), \ldots, P_{js_j}(\ul n)$.
\end{itemize}
\end{proposition}

Note that $c_p$ can also be computed with the following formula: 
$$c_p = \# \  \bigcup_{j=1}^\ell \left\{ \ul n \in (\Zz/p\Zz)^r \mid P_{j1}(\ul n) = 0 \pmod p, \ldots,P_{js_j}(\ul n) = 0 \pmod p \right\}.$$
This equality follows from the relation $V(I\cdot J)=V(I) \cup V(J)$ for ideals and their varieties, applied to $\Pi = {\mathcal P}_1 \cdots {\mathcal P}_\ell$.

\begin{proof} 
(a) Assume that some irreducible polynomial $D\in \Qq[\ul x]$ divides all elements of $\Pi$. Then the product of all ideals $\langle {\mathcal P}_j\rangle \subset \Qq[\ul x]$ ($j=1,\ldots,\ell$), which 
is generated by the set $\Pi$, is contained in the ideal $\langle D \rangle \subset \Qq[\ul x]$. As $\langle D \rangle$ is a prime ideal, 
we have  $\langle {\mathcal P}_j\rangle\subset \langle D \rangle$ for some $j\in \{1,\ldots,\ell\}$. This contradicts $P_{j1}(\ul x),\ldots,P_{js_j}(\ul x)$ being coprime.
\smallskip

(b) $\mathcal{R}(\Pi) \subset \mathcal{R}$: If $\underline n\notin \mathcal{R}$, i.e., $\underline n\notin\mathcal{R}({\mathcal P}_j)$ for some $j\in \{1,\ldots,\ell\}$, then some prime $p$ divides $P_{j1}(\ul n),\ldots, P_{js_\ell}(\ul n)$.
Clearly then, $p$ divides all $Q(\ul n)$ with $Q\in \Pi$, i.e., $\underline n\notin \mathcal{R}(\Pi)$.
\smallskip

$\mathcal{R}(\Pi) \supset \mathcal{R}$: Let $\underline n\notin \mathcal{R}(\Pi)$, i.e., some prime $p$ divides all $Q(\underline n)$ with $Q\in \Pi$. Observe that 
 the ideal generated by all these $Q(\ul n)$ is the product of the ideals $\langle P_{j1}(\ul n),\ldots, P_{js_j}(\ul n)\rangle \subset \Zz$ with $j$ ranging over $\{1,\ldots, s\}$.
 So this product is contained in $p\Zz$. But then $p\Zz$ must contain some 
ideal $\langle P_{j1}(\ul n),\ldots, P_{js_j}(\ul n)\rangle$; hence $\underline n\notin\mathcal{R}({\mathcal P}_j)$ and so $\underline n \notin \mathcal{R}$.
\smallskip

(c) Follows from (b) and the Ekedahl--Poonen formula, for the case $\ell=1$, with the polynomials in $\Pi$ (Theorem \ref{th:poonen}).
\smallskip

(d) We have $c_p=p^r$ if and only if $p$ divides all $Q(\ul n)$ with $Q\in \Pi$ for every $\ul n\in \Zz^r$. Arguing as in (b) above for each fixed $\ul n\in \Zz^r$, we obtain that,
for each $\ul n$, $p$ divides $P_{j1}(\ul n),\ldots, P_{js_j}(\ul n)$ for some $j\in \{1,\ldots,\ell\}$, which is the claimed condition. The converse is clear.
\end{proof}

Finally we can give the proof of Theorem \ref{cor:HS} in the general case $\ell \geq 1$.

\begin{proof}[Proof of Theorem \ref{cor:HS} ($\ell \geq 1$)] 
Let $P_1(\underline x,\underline y), \ldots, P_\ell(\underline x,\underline y)$ be as in the statement.
The first point is based on the same result of Cohen used in the case $\ell=1$. Specifically let ${\mathcal H}(P_1,\ldots,P_\ell)$ 
be the subset of $\Zz^r$ of all $\underline n$ such that $P_1(\underline n,\underline y), \ldots, P_\ell(\underline n,\underline y)$ 
are irreducible in $\Qq [\underline y]$. From Theorem 1 of \cite[\S 13]{serreMW},  we have $\tilde\mu({\mathcal H}(P_1,\ldots,P_\ell))=1$,
with $\tilde \mu$ the density introduced in Remark \ref{rem:density}.

For each $j=1,\ldots,\ell$, denote by ${\mathcal P}_j\subset \Qq[\ul x]$ the set of coefficients $P_{j1}(\ul x),\ldots, P_{js_j}(\ul x)$ of $P_j$, viewed as a polynomial in $\underline y$; 
these polynomials are coprime. Using then the notation of Proposition \ref{prop:poonen-gen}, the set ${\mathcal R}\subset \Zz^r$ is the subset of all $\ul n$ such that the 
polynomials $P_1(\underline n,\underline y), \ldots, P_\ell(\underline n,\underline y)$ are primitive.
Thus, for every $\ul n\in H={\mathcal H}(P_1,\ldots,P_\ell) \cap \textstyle \mathcal{R}$, the polynomials $P_1(\underline n,\underline y), \ldots, P_\ell(\underline n,\underline y)$ 
are irreducible in $\Zz[\ul y]$. 

Observe that the assumption that there is no prime $p$ such that $\prod_{j=1}^\ell P_j(\underline n,\underline y) \equiv 0 \pmod{p}$ for every $\underline n \in \Zz^r$ forbids
the equivalent conditions from Proposition \ref{prop:poonen-gen}(d) to happen. Thus, by Proposition  \ref{prop:poonen-gen}(c), we have $\mu(\mathcal{R})>0$, and also
$\tilde \mu(\mathcal{R})>0$ (as explained in Remark \ref{rem:density}). Conclude that $\tilde \mu(H)>0$ as well.
\end{proof}


\bibliographystyle{plain}
\bibliography{ngcd16.bib}

\end{document}